\documentclass[11pt]{amsart}
\pdfoutput=1
\usepackage{heuristica}
\usepackage[heuristica,vvarbb,bigdelims]{newtxmath}

\usepackage{fullpage,url,mathrsfs,amssymb,color,multicol,enumerate,amscd}
\usepackage{multirow}
\usepackage{mathtools}
\usepackage[all]{xy} 
\usepackage{verbatim,moreverb}

\usepackage[OT2,T1]{fontenc}

\usepackage{tikz-cd}
\usepackage{xfrac}
\usepackage{float}

\newcommand\paren[1]{\left(#1\right)}
\newcommand\set[1]{\left\{#1\right\}}
\newcommand\abs[1]{\left| #1 \right|}

\newcommand{\F}{{\mathbb F}}  
\newcommand{\Q}{{\mathbb Q}}
\newcommand{\ZZ}{{\mathbb Z}}
\newcommand{\Zhat}{{\widehat{\ZZ}}}
\newcommand{\sqf}{{\operatorname{sf}}}

\newcommand{\TwoCs}{{\normalfont \texttt{2Cs}}}
\newcommand{\TwoB}{{\normalfont \texttt{2B}}}
\newcommand{\TwoCn}{{\normalfont \texttt{2Cn}}}
\newcommand{\Attt}{{\normalfont \texttt{A}}}
\newcommand{\Bttt}{{\normalfont \texttt{B}}}

\newcommand{\eight}{{\normalfont\texttt{2.6.0.1}}}
\newcommand{\eighta}{{\normalfont\texttt{8.12.0.2}}}
\newcommand{\eightb}{{\normalfont\texttt{4.12.0.2}}}
\newcommand{\eightc}{{\normalfont\texttt{8.12.0.1}}}
\newcommand{\eightd}{{\normalfont\texttt{4.12.0.1}}}
\newcommand{\thirtyeight}{{\normalfont\texttt{8.12.0.3}}}
\newcommand{\thirtyeighta}{{\normalfont\texttt{8.24.0.5}}}
\newcommand{\thirtyeightb}{{\normalfont\texttt{8.24.0.7}}}
\newcommand{\thirtyeightc}{{\normalfont\texttt{8.24.0.2}}}
\newcommand{\thirtyeightd}{{\normalfont\texttt{8.24.0.1}}}
\newcommand{\fortysix}{{\normalfont\texttt{8.12.0.4}}}
\newcommand{\fortysixa}{{\normalfont\texttt{8.24.0.6}}}
\newcommand{\fortysixb}{{\normalfont\texttt{8.24.0.8}}}
\newcommand{\fortysixc}{{\normalfont\texttt{8.24.0.3}}}
\newcommand{\fortysixd}{{\normalfont\texttt{8.24.0.4}}}
\newcommand{\six}{{\normalfont\texttt{2.3.0.1}}}
\newcommand{\fourteen}{{\normalfont\texttt{8.6.0.2}}}
\newcommand{\fifteen}{{\normalfont\texttt{8.6.0.4}}}
\newcommand{\sixteen}{{\normalfont\texttt{8.6.0.1}}}
\newcommand{\seventeen}{{\normalfont\texttt{8.6.0.6}}}
\newcommand{\eighteen}{{\normalfont\texttt{8.6.0.3}}}
\newcommand{\nineteen}{{\normalfont\texttt{8.6.0.5}}}
\newcommand{\two}{{\normalfont\texttt{2.2.0.1}}}
\newcommand{\twoa}{{\normalfont\texttt{4.4.0.2}}}
\newcommand{\twob}{{\normalfont\texttt{8.4.0.1}}}

\DeclareMathOperator{\Frob}{Frob}
\DeclareMathOperator{\Quo}{{Quo}}
\DeclareMathOperator{\lcm}{lcm}
\DeclareMathOperator{\Aut}{Aut}
\DeclareMathOperator{\Gal}{Gal}
\DeclareMathOperator{\GL}{GL}
\DeclareMathOperator{\SL}{SL}

\DeclareMathOperator{\ord}{ord}

\newtheorem{theorem}{Theorem}[section]
\newtheorem{proposition}[theorem]{Proposition}
\newtheorem{lemma}[theorem]{Lemma}
\newtheorem{corollary}[theorem]{Corollary}

\theoremstyle{definition}
\newtheorem{definition}[theorem]{Definition}
\newtheorem{example}[theorem]{Example}

\theoremstyle{remark}
\newtheorem{remark}[theorem]{Remark}

\numberwithin{equation}{section}

\begin{document}

\title{Serre curves relative to obstructions modulo 2}

\author{Jacob Mayle}
\address{Department of Mathematics, Wake Forest University, Winston-Salem, NC 27109}
\email{maylej@wfu.edu}

\author{Rakvi}
\address{Department of Mathematics, University of Pennsylvania, Philadelphia, PA 19104}
\email{rakvi@sas.upenn.edu}


\date{\today}

\begin{abstract} We consider elliptic curves $E / \mathbb{Q}$ for which the image of the  adelic Galois representation $\rho_E$ is as large as possible given a constraint on the image modulo 2. For such curves, we give a characterization in terms of their $\ell$-adic images, compute all examples of conductor at most 500,000, precisely describe the image of $\rho_E$, and offer an application to the cyclicity problem. In this way, we generalize some foundational results on Serre curves.
\end{abstract}

\maketitle

\section{Introduction}

Let $E/K$ be an elliptic curve defined over a number field $K$. Fix an algebraic closure $\overline{K}$ of $K$. Consider the adelic Galois representation of $E$, which is a continuous homomorphism of profinite groups,
\[ \rho_E \colon \Gal(\overline{K}/K) \longrightarrow \Aut(T(E)) \]
 that encodes the natural action of $\Gal(\overline{K}/K)$ on the adelic Tate module $T(E) \coloneqq \varprojlim E[n]$ of $E$. We can identify $\Aut(T(E))$ with $\GL_2(\Zhat)$ by choosing a $\Zhat$-basis for $T(E)$. We will denote the image of $\rho_E$ in $\GL_2(\Zhat)$ by $G_E$, which we understand to be only defined up to conjugation in $\GL_2(\widehat{\ZZ})$. In a celebrated 1972 article \cite{MR387283}, Serre proved that if $E$ does not have complex multiplication, then $G_E$ is an open subgroup of $\GL_2(\widehat{\ZZ})$, and hence is of finite index. In this sense, $G_E$ is a ``large'' subgroup of $\GL_2(\widehat{\ZZ})$.

In fact, Greicius \cite{MR2778661} constructed a cubic number field $K$ and an elliptic curve $E/K$ such that $\rho_E$ is surjective. On the other hand, if $E$ is defined over $\Q$, then Serre noted \cite[Proposition 22]{MR387283} that by the Weil pairing and Kronecker--Weber theorem,
\begin{equation} \label{E:index-2}[\GL_2(\widehat{\ZZ}) : G_E] \geq 2. \end{equation}
Throughout the rest of this paper, we restrict our attention to elliptic curves defined over $\Q$. A \emph{Serre curve} is an elliptic curve $E/\Q$ for which $G_E$ is maximal, in the sense that equality holds in \eqref{E:index-2}. Such curves are abundant. Indeed, Jones \cite{MR2563740} proved that 100\% of elliptic curves over $\Q$ are Serre curves in a suitable sense of density. Serre curves have many applications (e.g., \cite{MR3204298,Nunodensity, MR4041152, MR2534114, MR3068548, LombardoTronto, MR4271815, MR2805578}) largely because the adelic image $G_E$ of such a curve is readily known (it depends only on the discriminant $\Delta_E$ of $E$).

Despite their abundance, until recently, only a few examples of Serre curves appeared in the literature. In 2015, Daniels \cite{MR3349445} exhibited an infinite family, proving (for instance) that if $\ell$ is a prime number such that $\ell \not\in\{2,7\}$, then the elliptic curve given by
\[
y^2 + xy = x^3 + \ell
\]
is a Serre curve. In his Ph.D. thesis \cite{MayleThesis}, the first author gave an algorithm that determines whether a given elliptic curve is a Serre curve. The algorithm is implemented in Sage. Running it on Cremona's database \cite{CremonaDB} (accessed via the LMFDB \cite{LMFDB}) reveals that  1,477,879  of the 3,064,705 ($\approx$48.223\%) curves of conductor at most 500,000 are Serre curves.

There is considerable interest in understanding the adelic image $G_E$ of an elliptic curve $E$. Mazur \cite{MR0450283} articulated the overarching problem as follows: Given a subgroup $G \subseteq \GL_2(\widehat{\ZZ})$, classify all elliptic curves $E/\Q$ such that the inclusion $G_E \subseteq G$ holds\footnote{In fact, Mazur posed this question for elliptic curves over number fields.}. This problem is known as ``Mazur's Program B'' and Zywina \cite{ZywinaOpen} recently made a breakthrough on a computational variant of it. He gave an algorithm that, given a non-CM elliptic curve $E/\Q$ computes the adelic image $G_E$. Zywina's recent progress  follows an earlier theoretical algorithm of Brau \cite{BrauThesis}. It also follows a vast body of work that aims to understand the images of $\ell$-adic and residual representations, which we now briefly discuss.

For a prime number $\ell$ and an integer $n \geq 2$, consider the $\ell$-adic and mod $n$ Galois representations of $E$,
\begin{align*}
    \rho_{E,\ell^\infty} \colon \Gal(\overline{\Q}/\Q) \longrightarrow \GL_2(\ZZ_\ell)
    \quad \text{and} \quad
    \rho_{E,n} \colon \Gal(\overline{\Q}/\Q) \longrightarrow \GL_2(\ZZ/n\ZZ).
\end{align*}
These representations encode the natural action of $\Gal(\overline{\Q}/{\Q})$ on the $\ell$-adic Tate module $T_\ell(E)$ and $n$-torsion subgroup $E[n]$ of $E$, respectively. Alternatively, they can be viewed as the composition of the adelic Galois representation $\rho_E$ with the natural projections $\GL_2(\widehat{\ZZ}) \to \GL_2(\ZZ_\ell)$ and $\GL_2(\widehat{\ZZ}) \to \GL_2(\ZZ/n\ZZ)$, respectively. We write $G_E(\ell^\infty)$ and $G_E(n)$ for the images of $\rho_{E,\ell^\infty}$ and $\rho_{E,n}$, respectively. A second perspective on Serre's open image theorem \cite{MR387283} is that if $E$ does not have complex multiplication, then $\rho_{E,\ell^\infty}$ (and hence $\rho_{E,\ell}$) is surjective for all but finitely many prime numbers $\ell$.

Assume that $E$ does not have complex multiplication. In 2015, Zywina \cite{ZywinaSurj} gave an algorithm that computes a finite set of primes outside of which $\rho_{E,\ell}$ is surjective. In the same year, Zywina \cite{ZywinaImages} and Sutherland \cite{MR3482279} gave  algorithms that compute the image of $\rho_{E,\ell}$ for any prime number $\ell$. Also that year,  Rouse--Zureick-Brown \cite{MR3500996} gave an algorithm that computes the image of $\rho_{E,2^\infty}$. In 2021, Rouse--Sutherland--Zureick-Brown \cite{RSZB} gave an algorithm that computes the image of $\rho_{E,\ell^\infty}$ for a given prime number $\ell$. There is also a growing body of research on entanglements \cite{MR3447646,DanielsLozanoRobledo,DLRM,DanielsMorrow,MR4374148, MR3957898} that aims to understand images of mod $n$ representations for $n$ a product of at least two distinct primes.

In this paper, we study elliptic curves $E/\Q$ for which the adelic image $G_E$ is maximal relative to a prescribed obstruction. Roughly speaking, for an integer $n \geq 2$ and a subgroup $G \subseteq \GL_2(\ZZ/n\ZZ)$, an elliptic curve $E/\Q$ is a \textit{$G$-Serre curve} if $G_E$ is ``as large as possible'' given the constraint that $G_E(n) \subseteq G$.
This notion of a relative Serre curve was originally set forth by Jones in \cite{MR2563740}. We give a proper definition and discuss the notion further in \S\ref{S:Characterization}.

More specifically, in this paper we consider $G$-Serre curves for proper subgroups $G \subseteq \GL_2(\ZZ/2\ZZ)$. The group $\GL_2(\ZZ/2\ZZ)$ is isomorphic to the symmetric group on 3 letters. Thus it has 3 proper subgroups up to conjugacy. These subgroups are of index 6, 3, and 2, and we denote them by $\TwoCs \coloneqq \left\{ \begin{psmallmatrix} 1 & 0 \\ 0 & 1 \end{psmallmatrix} \right\}$, $\TwoB \coloneqq \left\langle \begin{psmallmatrix} 1 & 1 \\ 0 & 1 \end{psmallmatrix} \right\rangle$, and $\TwoCn \coloneqq \left\langle \begin{psmallmatrix} 0 & 1 \\ 1 & 1 \end{psmallmatrix} \right\rangle$, respectively.

Our main theorem classifies $G$-Serre curves for subgroups $G \subseteq \GL_2(\ZZ/2\ZZ)$. It relies crucially on the work of Rouse--Zureick-Brown on 2-adic images. This work was later subsumed by the work of Rouse--Sutherland--Zureick-Brown \cite{RSZB}, which considers $\ell$-adic images more generally. Thus, in what follows, we write  \texttt{A.B.C.D} for the subgroup of $\GL_2(\ZZ_2)$ with the given RSZB label. Define the sets
\begin{align*}
\mathcal{S}_{{\TwoCs}} &\coloneqq \{\eight,\eighta,\eightb,\eightc,\eightd,\thirtyeight, \\
&\qquad\; \thirtyeighta,\thirtyeightb,\thirtyeightc,\thirtyeightd,\fortysix,\fortysixa, \\
&\qquad\; \fortysixb,\fortysixc, \fortysixd\}, \\
\mathcal{S}_{{\TwoB}} &\coloneqq \{\six, \fourteen,\fifteen,\sixteen,\seventeen,\eighteen,\nineteen\}, \\
\mathcal{S}_{{\TwoCn}} &\coloneqq \{\two,\twoa,\twob\}.
\end{align*}
With this notation in place, we now state our classification result.

\begin{theorem} \label{T:MainTheorem} Let $G \in \{{\TwoCs},{\TwoB},{\TwoCn}\}$ and $E/\Q$ be an elliptic curve. We have that $E$ is a $G$-Serre curve if and only if  $G_E(2^\infty) \in \mathcal{S}_{G}$ and $\rho_{E,\ell^\infty}$ is surjective for each odd prime $\ell$.
\end{theorem}

We prove Theorem \ref{T:MainTheorem} in Section \ref{S:RelSerrCrvs}. Together with the work of Rouse--Sutherland--Zureick-Brown \cite{RSZB_GitHub,RSZB}, it yields an algorithm for determining whether a given elliptic curve is a $G$-Serre curve for a proper subgroup $G \subseteq \GL_2(\ZZ/2\ZZ).$ \footnote{Note that to check the surjectivity of $\rho_{E,\ell^\infty}$, it suffices to check the surjectivity of $\rho_{E,\ell^k}$ where $k = 3$ if $\ell = 2$, $k = 2$ if $\ell = 3$, and $k = 1$ if $\ell \geq 5$ (see Lemma \ref{L:Lift}).}
We implemented this algorithm in Magma \cite{MR1484478}. The code can be found in this paper's GitHub repository \cite{MayleRakviGitHub}:

\centerline{\url{https://github.com/maylejacobj/RelativeSerreCurves}}

\noindent There are 3,064,705 elliptic curves over $\Q$ of conductor at most 500,000 \cite{CremonaDB}. We find that of these, 83,637 ($\approx$2.729\%) are \TwoCs-Serre curves, 827,120 ($\approx$26.989\%) are \TwoB-Serre curves, and 4,122 ($\approx$0.134\%) are \TwoCn-Serre curves. In total, 2,392,758 ($\approx$78.075\%) curves of conductor at most 500,000 are either Serre curves or $G$-Serre curves for a proper subgroup $G \subseteq \GL_2(\ZZ/2\ZZ)$. All of our computations were run on a machine with a 2.9 GHz 6-Core Intel Core i9 processor and 32 GB of memory.

In Section \ref{S:descriptions}, we describe the adelic images of the considered $G$-Serre curves. In Section \ref{S:Application}, we give an application of the above work to the cyclicity problem and briefly discuss other applications. In Section \ref{S:examples}, we provide three detailed examples where we compute generators for the adelic Galois image using results of Section \ref{S:descriptions}. Finally, in the appendix we give a table of minimal conductor $G$-Serre curves for each 2-adic image appearing in Theorem \ref{T:MainTheorem}, including each curve's image conductor and cyclicity correction factor.

\subsection*{Acknowledgments} We thank Tian Wang and Nathan Jones for reading an earlier version of this paper and providing their useful comments. We would also like to thank the anonymous referees for their helpful comments and suggestions. 

\section{Characterization of relative Serre curves} \label{S:Characterization}

In this section, we give a  characterization of Serre curves relative to obstructions modulo 2. We begin by defining some notation and terminology, and giving some preliminaries.

Throughout this paper, $m$ and $n$ denote positive integers and $\ell$ denotes a prime number. We write $\ZZ_\ell$ to denote the ring of $\ell$-adic integers and $\widehat{\ZZ}$ for the ring of profinite integers. For a subgroup $G \subseteq \GL_2(\widehat{\ZZ})$, we write $G(\ell^\infty)$ and $G(n)$ for the images of $G$ under the natural projections $\GL_2(\widehat{\ZZ}) \to \GL_2(\ZZ_\ell)$ and $\GL_2(\widehat{\ZZ}) \to \GL_2(\ZZ/n\ZZ)$, respectively. In the same way, if $G \subseteq \GL_2(\ZZ/n\ZZ)$ is a subgroup and $m$ divides $n$, then $G(m)$ denotes the image of $G$ under the natural projection $\GL_2(\ZZ/n\ZZ) \to \GL_2(\ZZ/m\ZZ)$. For a subgroup $G \subseteq \GL_2(\ZZ/n\ZZ)$, we write $\widehat{G}$ for the preimage of $G$ under the natural projection $\GL_2(\widehat{\ZZ}) \to \GL_2(\ZZ/n\ZZ)$. Similarly if $G \subseteq \GL_2(\ZZ_\ell)$, then $\widehat{G}$ denotes the preimage of $G$ under $\GL_2(\widehat{\ZZ}) \to \GL_2(\ZZ_\ell)$. For any profinite group $G$, we write $[G,G]$ to denote the closure of the commutator subgroup of $G$.

Let $E/\Q$ be an elliptic curve. Let $n \geq 2$ be an integer and $G \subseteq \GL_2(\ZZ/n\ZZ)$ be a subgroup. We now define the notion of a $G$-Serre curve. Our definition is readily seen to be equivalent to the definition given by Jones \cite{MR3350106}, though our choice of notation is a bit better suited for our purposes. The definition (here and in \cite{MR3350106}) is in terms of commutators, which offers a tractable condition to check in practice.
\begin{definition} \label{D:G_Serre} We say that an elliptic curve $E/\Q$ is a   \emph{$G$-Serre curve} if $G_E(n) \subseteq G$ and $[G_E,G_E] = [\widehat{G},\widehat{G}]$.
\end{definition}

\begin{remark} It is possible that $E$ is a $G$-Serre curve yet $G_E(n)$ is a proper subgroup of $G$. For instance, consider the elliptic curve $E$ with LMFDB label 200.a1 given by the Weierstrass equation
\[ y^2=x^3+125x-1250. \]
From the data provided in the ``Galois representations'' section of the curve's LMFDB page, we see that $E$ is a Serre curve. The property of an elliptic curve being a Serre curve is equivalent to it being a $\GL_2(\ZZ/n\ZZ)$-Serre curve for any $n$. In particular, $E$ is a $\GL_2(\ZZ/8\ZZ)$-Serre curve. However, the mod $8$ Galois representation of $E$ is nonsurjective, i.e., $G_E(8)$ is a proper subgroup of $\GL_2(\ZZ/8\ZZ)$. While the phenomenon of $G_E(n) \subsetneq G$ is possible in general, we will see in Proposition \ref{P:adelic_index} that if $E$ is a $G$-Serre curve for some $G \in \{{\TwoCs},{\TwoB},{\TwoCn}\}$, then $G_E(2) = G$.
\end{remark}

Next, we state a lemma that provides two useful properties of $G_E$ that follow from the Weil pairing and Kronecker--Weber theorem. In order to state it, we first recall some  terminology.  Let $H \subseteq \GL_2(\widehat{\ZZ})$ be a subgroup. We say that $H$ is \emph{determinant-surjective} if $\det(H) = \widehat{\ZZ}^\times$. We say that $H$ is  \emph{commutator-thick} if $[H,H] = H \cap \SL_2(\widehat{\ZZ})$.

\begin{lemma}\label{L:DSCT} If $E/\Q$ is an elliptic curve, then $G_E$ is determinant-surjective and commutator-thick.
\end{lemma}
\begin{proof} See, for instance, \cite[Remark 2.3]{MR3350106} or \cite[\S 1.2]{ZywinaOpen}. 
\end{proof}

The next lemma appears in  \cite[Remark 2.5]{MR3350106} and is less well-known. It offers a second perspective on $G$-Serre curves in terms of the index of $G_E$. We relay its proof below.

\begin{lemma} \label{L:eq_def} If $E$ is a $G$-Serre curve, then the index $[\widehat{G}: G_E]$ is minimal among the indices of determinant-surjective and commutator-thick subgroups of $\widehat{G}$.
\end{lemma}
\begin{proof}
Let $H$ be a determinant-surjective and commutator-thick subgroup of $\widehat{G}$. To prove the lemma, we need to show that $[\widehat{G} : G_E] \le [\widehat{G} : H].$ Since $H$ is determinant-surjective we have the exact sequence for $H$,
\[
1 \xrightarrow{\phantom{\det}} H \cap \SL_2({\widehat{\ZZ}}) \xrightarrow{\phantom{\det}} H \xrightarrow{\det} \widehat{\ZZ}^\times \xrightarrow{\phantom{\det}} 1.
\] As $H \subseteq \widehat{G}$, we have that $[H,H] \subseteq [\widehat{G},\widehat{G}]$. Consequently,
\[ [\widehat{G} : H] =[\widehat{G} \cap \SL_2(\widehat{\ZZ}) : H \cap \SL_2(\widehat{\ZZ})]= [\widehat{G} \cap \SL_2(\widehat{\ZZ}) : [H,H]] \geq [\widehat{G} \cap \SL_2(\widehat{\ZZ}) : [\widehat{G},\widehat{G}]].\] 
As $G_E$ is determinant-surjective, commutator-thick, and satisfies $[G_E,G_E]=[\widehat{G},\widehat{G}]$,
\begin{align*}
[\widehat{G} : G_E] &= [\widehat{G} \cap \SL_2(\widehat{\ZZ}) : G_E \cap \SL_2(\widehat{\ZZ})] \\
&= [\widehat{G} \cap \SL_2(\widehat{\ZZ}) : [G_E,G_E]] \\
&=[\widehat{G} \cap \SL_2(\widehat{\ZZ}) :[\widehat{G},\widehat{G}]].
\end{align*}
Hence $[\widehat{G} : G_E] \le [\widehat{G} : H]$, as needed.
\end{proof}

The perspective on $G$-Serre curves conveyed by Lemma \ref{L:eq_def} more closely reflects the standard definition of a Serre curve in terms of the index of $G_E$. This point-of-view will be valuable when describing adelic images in Section \ref{S:descriptions}.

\subsection{Reduction to a finite modulus} Let $H$ be a subgroup of an open subgroup $G \subseteq \GL_2(\widehat{\ZZ})$. In view of Definition \ref{D:G_Serre}, we turn to the problem of checking whether $[H,H] = [G,G]$ holds. Let $m$ be the \emph{level} (sometimes called the \emph{conductor}) of $G$, i.e., the least positive integer such that $G = \widehat{G(m)}$. As $G$ is open, such an $m$ must exist. The problem of checking whether $[H,H] = [G,G]$ can be simplified by the following theorem. When we state the result, we make an assumption on $m$ that by \cite[Remark 2.8]{MR3350106} allows us to use the value of $m_0$ as in the statement of the theorem rather than the larger value appearing in \cite[Equation 10]{MR3350106}.

\begin{theorem} \label{P:m0} Let $H \subseteq G$ be an open subgroup. There exists a constant $m_0$, depending on the group $G$, so that $[H,H] = [G,G]$ if and only if
\begin{enumerate}
	\item \label{L:m0-1} For each prime number $\ell \nmid m_0$, one has that $\SL_2(\ZZ/\ell\ZZ) \subseteq H(\ell)$ and
	\item \label{L:m0-2} One has that $[H(m_0),H(m_0)] = [G(m_0),G(m_0)]$.
\end{enumerate}
If the level $m$ of $G$ is such that each prime $\ell$ dividing $m$ satisfies $\ell \not\equiv \pm 1 \pmod{5}$, then $m_0$ may be taken to be the constant
\begin{equation} 
m_0 \coloneqq \lcm\left(2^3 \cdot 3^3, \prod_{\ell \mid m} \ell^{2\ord_\ell(m) + 1}\right) \label{def-m0} 
\end{equation}
where $\ord_{\ell}(m)$ denotes the exact power of $\ell$ dividing $m$.
\end{theorem}
\begin{proof} See \cite[Theorem 2.7 and Remark 2.8]{MR3350106}.
\end{proof}

The above theorem is the starting point for our work. We are interested in the case of $m = 2$, where \eqref{def-m0} gives a value of $m_0 = 216$ for the constant appearing in Theorem \ref{P:m0}. We can reduce this constant as follows.

\begin{lemma} \label{Stage1} In Theorem \ref{P:m0}, we can take $m_0$ to be given by
\[
m_0 \coloneqq
\begin{cases}
72 & G = \widehat{\TwoCs} \\
36 & G \in \set{\widehat{\TwoB},\widehat{\TwoCn}}. \\
\end{cases}
\]
\end{lemma}
\begin{proof} For $G \in \{\widehat{\TwoCs},\widehat{\TwoB},\widehat{\TwoCn}\}$, let $m_1$ be the least positive integer dividing $216$ such that for each maximal subgroup $K  \subsetneq [G(216),G(216)]$, we have $K(m_1) \subsetneq [G(m_1),G(m_1)]$. Our \texttt{LevelLower1} script (available in the paper's GitHub repository \cite{MayleRakviGitHub}) computes $m_1$ directly by iterating over all maximal subgroups $K  \subsetneq [G(216),G(216)]$. Running the script, we find that 
\[
m_1 =
\begin{cases}
72 & G \in \set{\widehat{\TwoCs}, \widehat{\TwoB}}\\
36 & G = \widehat{\TwoCn}. 
\end{cases}
\]
We can take $m_0$ to be $m_1$ in Theorem \ref{P:m0}. This follows from the fact that $m_1$ is a multiple of $6$ (which ensures that Theorem \ref{P:m0}\eqref{L:m0-1} is satisfied) and the definition of $m_1$ (which ensures that Theorem \ref{P:m0}\eqref{L:m0-2} is satisfied). In the case that  $G = \widehat{\TwoB}$, we can lower $m_0$ further. Indeed, let $m_2$ be the least positive integer dividing 72 such that among all subgroups $K$ of $G = \widehat{\TwoB}$,
\[
[K(m_2),K(m_2)] = [G(m_2),G(m_2)] \quad \implies \quad [K(72),K(72)] = [G(72),G(72)].
\]
Our \texttt{LevelLower2} script \cite{MayleRakviGitHub} iterates over subgroups of $G(72)$ to compute that $m_2 = 36$ when $G = \widehat{\TwoB}$. Similarly as above, we can take $m_0$ to be $m_2$ in Theorem \ref{P:m0}.
\end{proof}

We break the proof of Lemma \ref{Stage1} down into two stages. In practice, this is worthwhile because \texttt{LevelLower1} is much more time efficient than \texttt{LevelLower2}. 

\subsection{The group-theoretic characterization} In this subsection, we consider possible group-theoretic entanglements to give a criterion on whether $[G,G] = [H,H]$ holds in the cases that we are considering. To this end, we begin by recalling the notion of a fiber product. Let $G_1$, $G_2$, and $Q$ be groups. Let $\psi_1\colon G_1 \to Q$ and $\psi_2 \colon G_2 \to Q$ be surjective group homomorphisms. The \emph{fiber product of $G_1$ and $G_2$ by $\psi \coloneqq (\psi_1,\psi_2)$} is the subgroup
\[ G_1 \times_{\psi} G_2 \coloneqq \set{(g_1, g_2) \in G_1 \times G_2 : \psi_1(g_1) = \psi_2(g_2)} \subseteq G_1 \times G_2. \]
The group $G_1 \times_{\psi} G_2$ has the property that it surjects onto both $G_1$ and $G_2$ via the usual projection maps. Goursat's lemma gives that the only subgroups of $G_1 \times G_2$ that surject onto both factors are, in fact, fiber products (see, e.g., \cite[p. 75]{MR1878556} or \cite[\S1.2.2]{BrauThesis}).

We shall use the following lemma concerning commutators of fiber products.
\begin{lemma} \label{L:cyclic-comm} Let $G_1$, $G_2$, $Q$, and $\psi$ be as above. If $Q$ is cyclic, then
\[ 
[G_1 \times_{\psi} G_2, G_1 \times_{\psi} G_2]
=
[G_1, G_1] \times [G_2, G_2].
\]
\end{lemma}
\begin{proof} See \cite[p. 174]{MR0568299} and \cite[Lemma 2.3]{MR3447646}.
\end{proof}

To apply Lemma \ref{L:cyclic-comm} in our context, we will need to understand the possible groups that can appear as $G_1,G_2,$ and $Q$. Our next two lemmas help clarify this. We first set some notation. For a positive integer $n$ and a subgroup $G \subseteq \GL_2(\ZZ/n\ZZ)$, consider the set
\begin{align*}
   \mathcal{M}(G) &\coloneqq \{ H \subseteq G : \det(H) = \det(G) \text{ and } [H,H] = [G,G] \}.
\end{align*} 

Now define the following subgroups of $\GL_2(\ZZ/4\ZZ)$,
\begin{align*}
    K_1 &\coloneqq \left\langle 
    \begin{pmatrix} 3 & 0 \\ 0 & 1 \end{pmatrix},
    \begin{pmatrix} 1 & 1 \\ 1 & 0 \end{pmatrix} 
    \right\rangle, \\
    K_2 &\coloneqq \left\langle 
    \begin{pmatrix} 3 & 0 \\ 0 & 1 \end{pmatrix},
        \begin{pmatrix} 3 & 0 \\ 2 & 3 \end{pmatrix},
    \begin{pmatrix} 1 & 2 \\ 0 & 1 \end{pmatrix}
    \right\rangle,  \\
    K_3 &\coloneqq \left\langle 
    \begin{pmatrix} 3 & 0 \\ 0 & 1 \end{pmatrix},
        \begin{pmatrix} 3 & 0 \\ 2 & 3 \end{pmatrix},
    \begin{pmatrix} 1 & 2 \\ 0 & 3 \end{pmatrix}
    \right\rangle. 
\end{align*}
Observe that $K_1(2) = \TwoCn$ and $K_2(2) = K_3(2) = \TwoCs$.

The next lemma specifies the sets $\mathcal{M}(G)$ for the groups $G$ that will be of relevance.

\begin{lemma} \label{L:M_comp} Up to conjugacy, we have that
\begin{enumerate}
    \item\label{L:M_comp1} $\mathcal{M}(\GL_2(\ZZ/9\ZZ)) = \set{\GL_2(\ZZ/9\ZZ)}$
    \item \label{L:M_comp2} $\mathcal{M}(\widehat{\TwoCn}(4)) = \{ \widehat{\TwoCn}(4), K_1\}$ 
    \item\label{L:M_comp3}  $\mathcal{M}(\widehat{\TwoB}(4)) = \{ \widehat{\TwoB}(4)) \}$
    \item\label{L:M_comp4}  $\mathcal{M}(\widehat{\TwoCs}(8)) = \{ H \subseteq \widehat{\TwoCs}(8) : H(4) \in \{ \widehat{\TwoCs}(4), K_2, K_3 \} \text{ and } \det(H) = (\ZZ/8\ZZ)^\times \}.$
\end{enumerate}
\end{lemma}
\begin{proof} We have that \eqref{L:M_comp1} holds since if $H \in \mathcal{M}(\GL_2(\ZZ/9\ZZ))$, then
\[
\SL_2(\ZZ/9\ZZ) = [\GL_2(\ZZ/9\ZZ),\GL_2(\ZZ/9\ZZ)]= [H,H] \subseteq H \quad \text{and} \quad \det(H) = (\ZZ/9\ZZ)^\times.
\]
Thus by considering the determinant map $H \xrightarrow[]{\det} (\ZZ/9\ZZ)^\times$, we see that $H = \GL_2(\ZZ/9\ZZ)$. For parts \eqref{L:M_comp2}, \eqref{L:M_comp3}, and \eqref{L:M_comp4}, we compute the set $\mathcal{M}(G)$ using our \texttt{FullCommDet} script \cite{MayleRakviGitHub}.
\end{proof}

In the previous lemma, our need to consider $\mathcal{M}(\widehat{\TwoCs}(8))$ rather than just $\mathcal{M}(\widehat{\TwoCs}(4))$ can be traced to the fact that $[\widehat{\TwoCs}(4),\widehat{\TwoCs}(4)]$ is trivial. The set $\mathcal{M}(\widehat{\TwoCs}(8))$ is rather large compared to the other sets in Lemma \ref{L:M_comp}; it contains 15 subgroups up to conjugacy. 

Our next lemma describes the common quotients $Q$ that can appear for us in Lemma \ref{L:cyclic-comm}. For a group $G$, write $\Quo(G)$ to denote the set of \emph{all} isomorphism classes of quotients $G/N$ for a normal subgroup $N \trianglelefteq G$. Note that this definition differs somewhat from elsewhere in the literature. We write $0$ for the trivial group and $\ZZ/n\ZZ$ for the cyclic group of order $n$. 

\begin{lemma} \label{L:Quo} We have that
\begin{align*}
\Quo&(\GL_2(\ZZ/9\ZZ)) \cap \Quo(K) \\ &= 
\begin{cases}
    \{0, \ZZ/2\ZZ, \ZZ/3\ZZ, \ZZ/6\ZZ \} & K \in \mathcal{M}(\widehat{\TwoCn}(4)) \\
    \{0, \ZZ/2\ZZ \} & K \in \mathcal{M}(\widehat{\TwoB}(4)) \cup \mathcal{M}(\widehat{\TwoCs}(8)).
\end{cases}
\end{align*}
\end{lemma}
\begin{proof} We compute the intersections directly using our \texttt{QuoIntersection} script \cite{MayleRakviGitHub}.
\end{proof}

We now state and prove our main group-theoretic proposition. 
\begin{proposition} \label{P:GpThy} Let $G \in \{\widehat{\TwoCs},\widehat{\TwoB},\widehat{\TwoCn}\}$ and let $m_0$ be as in Lemma \ref{Stage1}. For a subgroup $H \subseteq G$ with $\det(H(m_0)) = (\ZZ/m_0\ZZ)^\times$, we have that 
\[
[H(m_0),H(m_0)] = [G(m_0),G(m_0)] \; \iff \; 
\begin{cases} H(9) = \GL_2(\ZZ/9\ZZ) \; \text{ and } \; \\
H(2^k) \in \mathcal{M}(G(2^k)) \end{cases}\]
where $k \in \set{2,3}$ is such that $m_0 = 2^k \cdot 9$.
\end{proposition}
\begin{proof} For the forward direction, note that by reducing modulo 9,
\[ [H(9),H(9)] = [G(9),G(9)] \quad \text{and} \quad \det(H(9)) = (\ZZ/9\ZZ)^\times. \]
Further, we have that $G(9) = \GL_2(\ZZ/9\ZZ)$. Thus $H(9) \in  \mathcal{M}(\GL_2(\ZZ/9\ZZ))$, so by Lemma \ref{L:M_comp}, $H(9) = \GL_2(\ZZ/9\ZZ)$. Similarly, upon reducing modulo $2^k$, we find that $H(2^k) \in \mathcal{M}(G(2^k))$.

Now consider the reverse direction. By Goursat's lemma, $H(m_0)$ can be written as a fiber product
\[ H(m_0) \cong \GL_2(\ZZ/9\ZZ) \times_\psi H(2^k) \]
with $\psi = (\psi_1,\psi_2)$ where $\psi_1 \colon \GL_2(\ZZ/9\ZZ) \to Q$ and $\psi_2 \colon H(2^k) \to Q$ are surjective homomorphisms onto a group $Q$. We have that $Q \in \Quo(\GL_2(\ZZ/9\ZZ)) \cap \Quo(H(2^k))$  is cyclic by Lemma \ref{L:Quo}. By Lemma \ref{L:cyclic-comm}, we conclude that $[H(m_0),H(m_0)] = [G(m_0),G(m_0)]$.
\end{proof}

\subsection{Relative Serre curves} \label{S:RelSerrCrvs} 

Proposition \ref{P:GpThy} shows (somewhat surprisingly) that only the $\ell$-adic Galois images of $E$ need to be considered when checking whether an elliptic curve $E/\Q$ is a Serre curve relative to an obstruction modulo 2. This is in contrast to the situation with usual Serre curves, where it is observed in  \cite{MR3447646} that whether $E$ is a Serre curve also depends on the mod $6$ Galois image of $E$.

In this subsection, we prove Theorem \ref{T:MainTheorem}, our characterization of Serre curves relative to obstructions modulo 2. We first give two lemmas that are used in the proof. The first is a lifting lemma for subgroups of $\GL_2(\ZZ_\ell)$ with surjective determinant.

\begin{lemma} \label{L:Lift} Let $\ell$ be a prime number and let $H \subseteq \GL_2(\ZZ_\ell)$ be a closed subgroup such that $\det(H) = \ZZ_\ell^\times$. Then 
 $H = \GL_2(\ZZ_\ell)$ if and only if $H(\ell^k) = \GL_2(\ZZ/\ell^k\ZZ)$ where $k = 3$ if $\ell = 2$, $k = 2$ if $\ell = 3$, and $k = 1$ if $\ell \geq 5$.
\end{lemma}
\begin{proof} This follows from \cite[Lemma 3, p. IV-23]{MR1043865} and \cite[Exercise 1b and 1c, p. IV-27]{MR1043865}. See also \cite[Corollary 2.13]{MR2778661}.
\end{proof}

Recall that the sets $\mathcal{S}_{{\TwoCs}}, \mathcal{S}_{{\TwoB}},$ and $\mathcal{S}_{{\TwoCn}}$ of subgroups of $\GL_2(\ZZ_2)$ were defined in \S1. Our next lemma elucidates the connection between these sets and the problem at hand.

\begin{lemma} \label{L:SG} For $G \in \{{\TwoCs},{\TwoB},{\TwoCn}\}$, we have that
\[
\mathcal{S}_G = \set{H \subseteq \GL_2(\ZZ_2) : H(2^k) \in \mathcal{M}(\widehat{G}(2^k)) \,  \text{ and } \, \exists E/\Q \, \text{ with } \, G_E(2^\infty) = H}
\]
where $k$ is as in the statement of Proposition \ref{P:GpThy}.
\end{lemma}
\begin{proof} Running our \texttt{ConstructSG} script \cite{MayleRakviGitHub} confirms that the claimed equality holds by iterating over all 1208 possible 2-adic images appearing in the database of Rouse--Zureick-Brown \cite{RZB_DB}.
\end{proof}

We now prove our main classification result.

\begin{proof}[Proof of Theorem \ref{T:MainTheorem}] Let $G \in \set{\TwoCs,\TwoB,\TwoCn}$. By Lemma \ref{L:SG}, if $H \in \mathcal{S}_G$, then in particular, $H(2) \subseteq G$. Thus we assume that $E$ is such that $G_E(2) \subseteq G$. Hence $E$ is a $G$-Serre curve if and only if $[G_E,G_E] = [\widehat{G},\widehat{G}]$. By Theorem \ref{P:m0} and Lemma \ref{Stage1}, this commutator condition holds if and only if 
\begin{enumerate}
	\item \label{Pf:1} For each prime number $\ell \geq 5$, one has that $\SL_2(\ZZ/\ell\ZZ) \subseteq G_E(\ell)$ and
	\item \label{Pf:2} One has that $[G_E(m_0),G_E(m_0)] = [G(m_0),G(m_0)]$.
\end{enumerate}
By Lemma \ref{L:Lift} and the fact that $G_E$ is determinant-surjective, we have that \eqref{Pf:1} is equivalent to $G_E(\ell^\infty) = \GL_2(\ZZ_\ell)$ for each prime number $\ell \geq 5$. By Proposition \ref{P:GpThy}, condition \eqref{Pf:2} is equivalent to $G_E(9) = \GL_2(\ZZ/9\ZZ)$ and $G_E(2^k) \in \mathcal{M}(G_E(2^k))$ where $k$ is as in the statement of the proposition. The former of these is equivalent, by Lemma \ref{L:Lift}, to $G_E(3^\infty) = \GL_2(\ZZ_3)$. The latter is equivalent, by Lemma \ref{L:SG}, to $G_E(2^\infty) \in \mathcal{S}_G$.
\end{proof}

The characterization that we just proved is well-suited for computations. For instance, if $E$ is in the LMFDB \cite{LMFDB}, then the data on $\ell$-adic images provided on the curve's page allows one to immediately decide, via Theorem \ref{T:MainTheorem}, whether $E$ is a $G$-Serre curve for some $G \in \{\TwoCs,\TwoB,\TwoCn\}$. More generally, our \texttt{IsRelSerreCurve} script \cite{MayleRakviGitHub} uses Theorem \ref{T:MainTheorem} with \cite{RSZB_GitHub,RSZB} to decide if $E$ is a $G$-Serre curve for such a $G$.

\subsection{The adelic index}

In this subsection, we determine the adelic index $[\GL_2(\widehat{\ZZ}) : G_E]$ for Serre curves $E/\Q$ relative to obstructions modulo 2. We start with a lemma.

\begin{lemma} \label{comm-index} For $G \in \{{\TwoCs},{\TwoB},{\TwoCn}\}$, we have that
\[
[\SL_2(\widehat{\ZZ}) : [\widehat{G},\widehat{G}]] =
\begin{cases}
	12 & G \in \set{{\TwoB},{\TwoCn}} \\
	48 & G \in {\TwoCs}.
\end{cases}
\]
\end{lemma}
\begin{proof} Let $G \in \{{\TwoCs},{\TwoB},{\TwoCn}\}$ and consider the subgroup $K \coloneqq [\widehat{G},\widehat{G}] \subseteq \SL_2(\widehat{\ZZ})$. By \cite[Lemma 7.4]{ZywinaOpen}, we have that  $K$ is the preimage of $K(4)$ under the natural projection $\SL_2(\widehat{\ZZ}) \to \SL_2(\ZZ/4\ZZ)$. It follows that
\[ [\SL_2(\widehat{\ZZ}) : K] = [\SL_2(\ZZ/4\ZZ) : K(4)]. \]
The right-hand side is straightforward to compute in Magma because $K(4) = [\widehat{G}(4),\widehat{G}(4)]$. We do so using our \texttt{CommIndx} script \cite{MayleRakviGitHub}, and obtain the claimed indices.
\end{proof}

Using Lemma \ref{L:eq_def} and Lemma \ref{comm-index}, we now give a second characterization of the considered relative Serre curves in terms of the adelic index.

\begin{proposition} \label{P:adelic_index} For $G \in \{{\TwoCs},{\TwoB},{\TwoCn}\}$, an elliptic curve $E/\Q$ is a $G$-Serre curve if and only if
\begin{equation} \label{index-bd}
G_E(2) = G \quad \text{and} \quad
[\GL_2(\widehat{\ZZ}) : G_E] = 
\begin{cases}
	12 & G \in \set{{\TwoB},{\TwoCn}} \\
	48 & G = {\TwoCs}.
\end{cases}
\end{equation}
\end{proposition}
\begin{proof} 

For any elliptic curve $E/\Q$, by Lemma \ref{L:DSCT} and the fact that \[G_E \cap \SL_2(\widehat{\ZZ}) = \ker(G_E \xrightarrow{\det} \widehat{\ZZ}^\times),\] we  observe that
\begin{equation} \label{SL2Indx} [\GL_2(\widehat{\ZZ}) : G_E] = [\SL_2(\widehat{\ZZ}) : G_E \cap \SL_2(\widehat{\ZZ})] = [\SL_2(\widehat{\ZZ}) : [G_E,G_E]]. \end{equation}
Suppose first that $E$ is a $G$-Serre curve. Then $[G_E, G_E] = [\widehat{G},\widehat{G}]$, so the claim about the index of $G_E$ follows by \eqref{SL2Indx} and Lemma \ref{comm-index}. The claim that $G_E(2) = G$ follows by Theorem \ref{T:MainTheorem} and noting that $H(2) = G$ for each $H \in \mathcal{S}_G$. For the reverse direction, note that $G_E(2) = G$ implies that $G_E \subseteq \widehat{G}$. Hence $[G_E,G_E] \subseteq [\widehat{G},\widehat{G}]$ and, in fact, equality holds by \eqref{SL2Indx} and Lemma \ref{comm-index}.
\end{proof}

\section{Explicit descriptions of adelic images} \label{S:descriptions}

In this section, we describe the adelic image $G_E$ of elliptic curves $E/\Q$ relative to obstructions modulo 2. We begin with some preliminaries on division fields. The \emph{$n$-division field} of $E$ is the field $\Q(E[n])$ obtained by adjoining to $\Q$ all $x$- and $y$-coordinates of points in the $n$-torsion subgroup $E[n]$ of $E$. Recall that $\Q(E[n])/\Q$ is Galois and
\[ G_E(n) \cong \Gal(\Q(E[n])/\Q). \]
Define $\zeta_n \coloneqq e^{\frac{2 \pi i}{n}}.$ Recall that by the Weil pairing on $E$, 
\begin{equation} \label{E:CyclicSubfield}
\Q(\zeta_n) \subseteq \Q(E[n]).
\end{equation}

We begin by recalling a result that precisely describes the $2^k$-division field of an elliptic curve $E/\Q$ for $k \in \set{1,2,3}$. Fix an equation for $E$ of the form
\begin{equation} \label{E:factored}
y^2 = (x - \alpha_1)(x-\alpha_2)(x-\alpha_3)
\end{equation}
for some $\alpha_1,\alpha_2,\alpha_3$ in the field $\overline{\Q}$  of  algebraic numbers. Fix $A_1,A_2,A_3 \in \overline{\Q}$ so that
\[
A_i^2 = \alpha_{i+1} - \alpha_{i+2}
\]
where $i$ is considered modulo $3$. Now set $B_1,B_2,B_3 \in \overline{\Q}$ so that
\begin{equation} \label{E:Bi2} B_i^2 = A_i(A_{i+1} + \zeta_4 A_{i+2}). \end{equation}

With notation in place, we now state the result on 2-power division fields.

\begin{theorem} \label{T:Yelton} We have the following.
\begin{enumerate}
    \item \label{T:Yelton1} $\Q(E[2]) = \Q(\alpha_1,\alpha_2,\alpha_3)$,
    \item \label{T:Yelton2} $\Q(E[4]) = \Q(E[2],\zeta_4,A_1,A_2,A_3)$, and
    \item \label{T:Yelton3} $\Q(E[8]) = \Q(E[4],\zeta_8,B_1,B_2,B_3)$.
\end{enumerate}
\end{theorem}
\begin{proof} Part \eqref{T:Yelton1} is well-known. See \cite[Theorem 1]{MR3687433} for parts \eqref{T:Yelton2} and \eqref{T:Yelton3}.
\end{proof}

\begin{corollary}\label{C:sqrt2}
If $\sqrt{2} \in \Q(E[4])$, then $[\Q(E[8]) : \Q(E[4])] \leq 8$.
\end{corollary}
\begin{proof} As $\zeta_4 \in \Q(E[4])$ and $\sqrt{2} \in \Q(E[4])$, we have that $\zeta_8 \in \Q(E[4])$. By \eqref{E:Bi2} and Theorem \ref{T:Yelton}, we have that  $B_i^2 \in \Q(E[4])$ for each $i \in \{1,2,3\}$. Thus $[\Q(B_1,B_2,B_3):\Q(E[4])] \leq 8$, so the claim follows by Theorem \ref{T:Yelton}.
\end{proof}

We will make extensive use of Theorem \ref{T:Yelton} and Corollary \ref{C:sqrt2} in this section. Let $\Delta_E$ denote the discriminant of any model of $E$. An immediate consequence of the theorem is that  $\sqrt{\Delta_E} \in \Q(E[2])$ and $\sqrt[4]{\Delta_E} \in \Q(E[4])$ since, up to 12th powers, we compute from \eqref{E:factored} that
\begin{equation} \label{E:DeltaE}
    {\Delta}_E = 16(\alpha_1 - \alpha_2)^2(\alpha_1 - \alpha_3)^2(\alpha_2 - \alpha_3)^2.
\end{equation}

In view of \eqref{E:CyclicSubfield}, it is useful to review an aspect of cyclotomic fields. Namely, the Kronecker--Weber theorem gives that if $K/\Q$ is abelian, then there exists an integer $n$ such that $K \subseteq \Q(\zeta_n)$. The \emph{conductor} of $K$, denoted $f(K)$, is the least positive integer $n$ with this property. Let $\Delta_K$ denote the discriminant of the ring of integers of $K$. If $K = \Q(\sqrt{d})$ is a quadratic field, then it is well-known that
\begin{equation} \label{E:cond_quad}
f(K) = \Delta_K = \begin{cases} 
d & d \equiv 1 \pmod{4} \\ 4d & \text{otherwise}.
\end{cases}
\end{equation}
If $K/\Q$ is a cyclic cubic extension, then by the conductor-discriminant formula \cite[7.4.13, 7.4.14]{MR1313719}, we have that
\begin{equation} \label{E:cond_cubic}
f(K) = \sqrt{\Delta_K}.
\end{equation}

In Theorem \ref{T:MainTheorem} we learned that entanglements need not be considered when checking if an elliptic curve is Serre curve relative to an obstruction modulo 2. However in this section, entanglements will play an essential role in our descriptions of $G_E$. As such, we recall some preliminaries on entanglement fields. Let $K_1$ and $K_2$ be Galois extensions of $\Q$. The compositum $K_1K_2 / \Q$ is Galois with
\begin{equation} \label{E:GalFiber}
\Gal(K_1K_2 / \Q) \cong \Gal(K_1/\Q) \times_{\psi} \Gal(K_2/\Q)
\end{equation}
where $\psi \coloneqq (\psi_1,\psi_2)$ with each map $\psi_i \colon \Gal(K_i/\Q) \to \Gal(K_1 \cap K_2 / \Q)$ given by restriction (see, e.g., \cite[Lemma 1.2.8]{BrauThesis}). The field $K_1 \cap K_2$ is called the \emph{entanglement field}. We refer to the degree $[K_1 \cap K_2 : \Q]$ as the degree of the entanglement between $K_1$ and $K_2$. By \eqref{E:GalFiber}, we have that
\begin{equation} \label{E:CompIndx}
[K_1 K_2 : \Q] = \frac{[K_1 : \Q][K_2 : \Q]}{[K_1 \cap K_2 : \Q]}.
\end{equation}

With these preliminaries in hand, we now turn to the problem of describing $G_E$.

\subsection{Serre curves relative to 2Cs}\label{subsec:2Cs}
Let $E$ be a \TwoCs-Serre curve. Fix an equation
\begin{equation} E \colon y^2 = (x-a)(x-b)(x-c) \label{E:2Cs} \end{equation}
with $a,b,c \in \ZZ$. By Proposition \ref{P:adelic_index}, we know that
\[ [\GL_2(\widehat{\ZZ}) : G_E] = 48.\] 
Recall that by Theorem \ref{T:MainTheorem}, we have $G_E(2^\infty) \in \mathcal{S}_{\TwoCs}$ where
\begin{align*}
\mathcal{S}_{{\TwoCs}} &= \{\eight,\eighta,\eightb,\eightc,\eightd,\thirtyeight, \\
&\qquad\; \thirtyeighta,\thirtyeightb,\thirtyeightc,\thirtyeightd,\fortysix,\fortysixa, \\
&\qquad\; \fortysixb,\fortysixc, \fortysixd\}.
\end{align*}
The above groups have index $6, 12,$ or $24$ in $\GL_2(\ZZ_2)$, as indicated by the second number in the RSZB label. Hence to understand $G_E$, it remains to account for an index of $8$ in the first case above, $4$ in the second, and 2 in the third case. As we shall see, entanglements are the source of the greater adelic index in all cases.

For a squarefree integer $N$, define
\begin{equation} \label{E:prime-notation}
N' \coloneqq
\begin{cases}
    N & N \equiv 1 \pmod{4} \\
    -N & N \equiv 3 \pmod{4} \\
    \frac{1}{2}N & N \equiv 2 \pmod{8} \\
    -\frac{1}{2}N & N \equiv 6 \pmod{8}.
\end{cases}
\end{equation}
Note that $N' \in \ZZ$ and that since $N$ is squarefree, we must have that $N \not\equiv 0 \pmod{4}$. Further the definition in \eqref{E:prime-notation} is such that $N' \equiv 1 \pmod{4}$ for any squarefree integer $N$. Therefore, $\sqrt{N'} \in \Q(\zeta_{|N'|})$ which will be used later.

Let $A \coloneqq a - b$, $B \coloneqq a - c$, and $C \coloneqq b - c$. Consider the following set 
\[ S\coloneqq\{\abs{A_\text{sf}},\abs{B_{\text{sf}}},\abs{ C_{\text{sf}}}, \abs{AB_{\text{sf}}}, \abs{AC_{\text{sf}}}, \abs{BC_{\text{sf}}},
\abs{ABC_{\text{sf}}}\} \setminus \{1, 2 \}
\]
where $n_\text{sf}$ denotes the squarefree part of an integer $n$ (and, for instance,  $AB_{\text{sf}} \coloneqq (AB)_\text{sf}$) and ``$\setminus$'' denotes the set difference.  Let $N_1$ be the least integer in $S$. If $[\GL_2(\ZZ_2):G_E(2^\infty)] \leq 12$, then let $N_2$ 
be the smallest integer in $S$ such that $N_1 \nmid N_2$. If $[\GL_2(\ZZ_2):G_E(2^\infty)] = 6$, then let $N_3$ be the smallest integer in $S$ such that $N_1 \nmid N_3$, $N_2 \nmid N_3$, and $(N_1 N_2)_{\text{sf}} \nmid N_3$. The proof of Lemma \ref{L:CsEnt} justifies why the $N_i$ can be chosen in this way. For each $i$, set
\begin{equation} \label{E:ki}
k_i = \begin{cases}
    3 & N_i \equiv 0 \pmod{2} \\
    2 & N_i \equiv 1 \pmod{2}. \\
\end{cases}
\end{equation}
To illustrate these definitions, we now provide a brief example.
\begin{example} Let $E$ be the elliptic curve with LMFDB label \texttt{9405.f2}, which is given by the equation
\[y^2 = (x+61)(x+118)(x-179).\]
We calculate that $S = \{ 15, 33, 55, 57, 95, 209, 3135 \}$. Then $N_1 = 15$, $N_2 = 33$, $N_3 = 57$, and $k_1=k_2=k_3=2$.
\end{example}

We now give a lemma that describes the entanglements associated with $G_E$.

\begin{lemma} \label{L:CsEnt} If $G_E(2^\infty) = \eight$, then for each $i \in \set{1,2,3}$, there is a quadratic entanglement between $\Q(E[2^{k_i}])$ and $\Q(E[\abs{N_i'}])$. If \[ G_E(2^\infty) \in \set{\eighta,\eightb, \eightc,\eightd, \thirtyeight, \fortysix},\] then for each $i \in \set{1,2}$, there is a quadratic entanglement between $\Q(E[2^{k_i}])$ and $\Q(E[\abs{N_i'}])$. For all the other cases, there is a quadratic entanglement between $\Q(E[2^{k_1}])$ and $\Q(E[\abs{N_1'}])$.
\end{lemma}
\begin{proof}
By Theorem \ref{T:Yelton}\eqref{T:Yelton2}, we have that
\[ \Q(E[4]) = \Q(\sqrt{-1},\sqrt{A_\sqf},\sqrt{B_\sqf},\sqrt{C_\sqf}). \] 
Thus the quadratic subfields of $\Q(E[4])$ are the following:
\begin{align*}
\Q(\sqrt{-1}), 
\Q(\sqrt{{A}_{\text{sf}}}), 
\Q(\sqrt{{B}_{\text{sf}}}), 
\Q(\sqrt{{C}_{\text{sf}}}),
\Q(\sqrt{{AB}_{\text{sf}}}), 
\Q(\sqrt{{AC}_{\text{sf}}}), \\ 
\Q(\sqrt{{BC}_{\text{sf}}}),
\Q(\sqrt{{ABC}_{\text{sf}}}), 
\Q(\sqrt{{-A}_{\text{sf}}}), 
\Q(\sqrt{{-B}_{\text{sf}}}), 
\Q(\sqrt{{-C}_{\text{sf}}}), \\
\Q(\sqrt{{-AB}_{\text{sf}}}),
\Q(\sqrt{{-AC}_{\text{sf}}}),
\Q(\sqrt{{-BC}_{\text{sf}}}), \text{ and }
\Q(\sqrt{{-ABC}_{\text{sf}}}).
\end{align*}
First assume that $G_E(2^\infty) = \eight$. We verify in Magma that $\eight(4)$ contains exactly $15$ index $2$ subgroups. Hence $\Q(E[4])$ has exactly $15 = 2^4 - 1$ quadratic subfields, so all of the above subfields are distinct. Moreover, since $\eight(8)$ is the full preimage of $\eight(4)$ under the reduction map $\GL_2(\ZZ/8\ZZ) \to \GL_2(\ZZ/4\ZZ)$ we know that $\sqrt{2} \not\in \Q(E[4])$ by Corollary \ref{C:sqrt2}. Thus we may indeed choose $N_1, N_2, N_3 \in S$ as specified above the lemma. By \eqref{E:cond_quad} and Theorem \ref{T:Yelton}, we observe that $\sqrt{{N_i'}} \in \Q(E[2^{k_i}]) \cap \Q(E[\abs{N_i'}])$ provides a quadratic entanglement.

If $H \in \set{\eighta,\eightc,\thirtyeight,\fortysix}$, then $H(4)$ still contains exactly $15$ index $2$ subgroups. Hence $\Q(E[4])$ has exactly $15 = 2^4 - 1$ quadratic subfields. But in this case, $\sqrt{2} \in \Q(E[4])$ is possible, so we may only choose $N_1, N_2 \in S$ as specified above the lemma. In all of the remaining cases, we calculate in a similar way that $\Q(E[4])$ has $7 = 2^3 - 1$ quadratic subfields. If $H \in \set{\eightb,\eightd}$, then $H(8)$ is the full preimage of $H(4)$, so we know that $\sqrt{2} \not \in \Q(E[4])$ by Corollary \ref{C:sqrt2}, and hence we may choose $N_1,N_2$ as specified. On the other hand, if $H \in \{\thirtyeighta,\thirtyeightb,\thirtyeightc,\thirtyeightd,\fortysixa,\fortysixb,\fortysixc,$ $\fortysixd\}$, then it may be $\sqrt{2} \in \Q(E[4])$, but as $\Q(E[4])$ contains $7$ quadratic subfields, we are still able to choose $N_1$ as specified.
\end{proof}

\subsubsection{Description of image}
We now describe the adelic image $G_E$ of a \TwoCs-Serre curve $E$. Let $N_i$ and $k_i$ be as in Lemma \ref{L:CsEnt}. For each $i$, let $G_i \subseteq G_E(2^{k_i})$ be such that $G_E(2^{k_i})/G_i \cong \Gal(\Q(\sqrt{N_i'})/\Q) \cong \{\pm 1\}$.  Let $\epsilon_i \colon \widehat{G_E(2^{\infty})} \to G_E(2^{\infty}) \to  G_E(2^{\infty})/G_i$ be the restriction of the map $\GL_2(\widehat{\ZZ}) \to \GL_2(\ZZ_2)$ followed by the natural map and $\chi_{N_i'}$ be the Dirichlet character given by the unique surjective homomorphism $\widehat{\ZZ}^\times \to \Gal(\Q(\sqrt{N_i'})/\Q) \cong \set{\pm 1}.$

\begin{proposition}\label{prop:descimages}
With notation as above, if $E$ is a \TwoCs-Serre curve, then
\[G_E= \bigcap_i \{A \in \widehat{G_E(2^\infty)}|\epsilon_i(A)=\chi_{N_i'}(\det(A))\}.\]
Here the intersection runs over $i$ such that $N_i$ is defined for $G_E(2^\infty)$.
\end{proposition}

\begin{proof}

Assume that $G_E(2^\infty) = \eight$. The inclusion 
\begin{equation} \label{PE:ClaimedEq}
    G_E \subseteq \bigcap\limits^3_{i=1}\{A \in \widehat{G_E(2^\infty)}|\epsilon_i(A)=\chi_{N_i'}(\det(A))\}
\end{equation} follows from the containments $\Q(\sqrt{N_i'}) \subseteq \Q(E[2^{k_i}])$ and $\Q(\sqrt{N_i'}) \subseteq \Q(E[\abs{N_i'}]).$ By Proposition \ref{P:adelic_index}, to establish that \eqref{PE:ClaimedEq} is actually an equality, it suffices to note that 
\begin{equation} \label{PE:2CsAdelicIndx} \left[ \GL_2(\Zhat) :
\bigcap\limits^3_{i=1}\{A \in \widehat{G_E(2^\infty)}|\epsilon_i(A)=\chi_{N_i'}(\det(A))\} \right] = 48. \end{equation}
Observe that $\widehat{G_E(2^\infty)}=G_E(2^\infty) \times \prod_{\ell~odd} \GL_2(\ZZ_\ell).$ As each $N_i'$ is odd for $i \in \{1,2,3\}$, we can think of $\{A \in \widehat{G_E(2^\infty)}|\epsilon_i(A)=\chi_{N_i'}(\det(A))\}$ as the fiber product of $G_E(2^\infty)$ and $\prod_{\ell~odd} \GL_2(\ZZ_{\ell})$ by $(\epsilon_i,\chi_{N_i'}\circ \det)$ which is index $2$ in $\widehat{G_E(2^\infty)}$. Indeed, as $N_i'$ are distinct, we have that
\[ \left[ \widehat{G_E(2^\infty)} \colon
\bigcap\limits^3_{i=1}\{A \in \widehat{G_E(2^\infty)}|\epsilon_i(A)=\chi_{N_i'}(\det(A))\} \right] = 8. \]
Moreover,
\[ \left[ \GL_2(\widehat{\ZZ}) :\widehat{G_E(2^\infty)} \right] 
= \left[ \GL_2(\ZZ_2) :G_E(2^\infty) \right] = 6. \]
Thus \eqref{PE:2CsAdelicIndx} holds as claimed. There are two remaining cases to consider, namely, when the index of $G_E(2^\infty)$ in $\GL_2(\ZZ_2)$ is $12$ and $24$. Both of these cases follow similarly to the case of $G_E(2^\infty) = \eight$ just considered, and hence are omitted.
\end{proof}

Let $m_E$ denote the \emph{image conductor} of $E$, i.e., the level of the group $G_E \subseteq \GL_2(\widehat{\ZZ})$. The image conductor of $E$ is a useful invariant of $E$ that plays a role in numerous applications, including those mentioned in Section \ref{S:Application}. We may use Proposition \ref{prop:descimages} to determine $m_E$ for a \TwoCs-Serre curve $E$. In order to do so, we first give a general lemma on the level of a closed subgroup of $\GL_2(\widehat{\ZZ})$.

\begin{lemma} \label{L:imgcondmult} Let $G \subseteq \GL_2(\widehat{\ZZ})$ be a closed subgroup. Let $m_1$ and $m_2$ be coprime positive integers with the property that if $d_1$ and $d_2$ are positive integers such that $d_i$ divides $m_i$ for $i \in \{1,2\}$ and if $(d_1,d_2) \neq (m_1,m_2)$, then $\widehat{G(d_1 d_2)} \supsetneq \widehat{G(m_1m_2)}$. Then $m_1m_2$ divides the level of $G$ as a subgroup of $\GL_2(\widehat{\ZZ})$.
\end{lemma}
\begin{proof} Write $m_G$ to denote the level of $G$ as a subgroup of $\GL_2(\widehat{\ZZ})$. Seeking a contradiction, suppose that $m_1 m_2$ does not divide $m_G$. For $i \in \{1,2\}$, let $d_i \coloneqq \gcd(m_G,m_i)$. Then $d_i \mid m_i$ for $i \in \{1,2\}$ and $(d_1,d_2) \neq (m_1,m_2)$, so by assumption
\[ \widehat{G(d_1 d_2)} \supsetneq \widehat{G(m_1m_2)}. \]
Let $\pi : \prod_{\ell \mid m_G} \GL_2(\ZZ_{\ell}) \to \GL_2(\ZZ/m_G\ZZ)$ be the natural map. Recall that
\[ G \cong \pi^{-1}(G(m_G)) \times \prod_{\ell \nmid m_G} \GL_2(\ZZ_{\ell}) \]
under the isomorphism $\GL_2(\widehat{\ZZ}) \to \prod_{\ell \mid m_G} \GL_2(\ZZ_{\ell}) \times \prod_{\ell \nmid m_G} \GL_2(\ZZ_{\ell})$. Further, note that as $m_1$ and $m_2$ are coprime, we have that $d_1d_2 = \gcd(m_G,m_1m_2)$. Hence
\[
G(m_1m_2) = \widehat{G(m_G)}(m_1m_2) = \widehat{G(d_1d_2)}(m_1m_2).
\]
But this implies that $\widehat{G(m_1m_2)} = \widehat{G(d_1d_2)}$, a contradiction.
\end{proof}

We now give the formula for the image conductor $m_E$ that follows from Proposition \ref{prop:descimages}. The formula involves the integers $N_i$ and the $2$-adic level and index of the $2$-adic Galois image of $E$ (which we recall are the first number $\Attt$ and second number $\Bttt$ in the RSZB label, respectively).

\begin{corollary} \label{cor:m_E} Let $E/\Q$ be a $\TwoCs$-Serre curve with $G_E(2^\infty)=$ \texttt{A.B.C.D}.
\begin{itemize}
    \item If $\texttt{B} = 6$ (i.e., $G_E(2^\infty) = \eight$), then\[
    m_E = \begin{cases}
        \lcm(4,\abs{N_1'},\abs{N_2'},\abs{N_3'}) & 2 \nmid N_1 N_2 N_3 \\
        \lcm(8,\abs{N_1'},\abs{N_2'},\abs{N_3'}) & 2 \mid N_1 N_2 N_3. \\
    \end{cases}
    \]
    
    \item If $\Bttt = 12$,  then\[
    m_E = \begin{cases}
        \lcm(4,\abs{N_1'},\abs{N_2'}) & 2 \nmid N_1 N_2  \, \text{ and } \, \Attt = 4 \\
        \lcm(8,\abs{N_1'},\abs{N_2'}) & 2 \mid N_1 N_2 \, \text{ or }\, \Attt = 8. \\
    \end{cases}
    \]
    
    \item If $\Bttt = 24$, then $m_E = \lcm(8,\abs{N_1'})$.
\end{itemize}
\end{corollary}

\begin{proof}
    We will consider the case that $G_E(2^\infty) = \eight.$ Let $k=2$ if $N_1N_2N_3$ is odd and $k=3$ otherwise. Let $\mathcal{L} \coloneqq \lcm(2^k,\abs{N_1'},\abs{N_2'},\abs{N_3'})$. Proposition \ref{prop:descimages} gives that $m_E$ divides $\mathcal{L}$. Let $N \in \{ N_1, N_2, N_3 \}$ and let $j = 2$ if $N$ is odd and $j = 3$ if $N$ is even. We will show that $2^j \cdot \abs{N'}$ divides $m_E$. 
    
    By the proof of Lemma \ref{L:CsEnt}, we have that $\sqrt{N'} \in \Q(E[2^j]) \cap \Q(E[\abs{N'}])$. It is clear that $\sqrt{N'} \not\in \Q(E[2])$ since $\Q(E[2]) = \Q$. If $j = 3$, then $\sqrt{N'} \not\in \Q(E[4])$ for otherwise $\sqrt{2} \in \Q(E[4])$ which is impossible by Corollary \ref{C:sqrt2} since $\Attt = 2$. Now suppose by way of a contradiction that there exists a proper divisor $d > 1$ of $\abs{N'}$ such that $\sqrt{N'} \in \Q(E[d])$. By the construction of $\{ N_1, N_2, N_3 \}$, we know that $\sqrt{d'} \notin \Q(E[2^j])$. We have that $\sqrt{d'} \in \Q(E[{d}])$, so  $\sqrt{N'/d'} \in \Q(E[d])$. Note that $N'/d' \equiv 1 \pmod{4}$, so also $\sqrt{N'/d'} \in \Q(E[\abs{N'/d'}])$. But then we have the entanglement $\sqrt{N'/d'} \in \Q(E[d]) \cap \Q(E[\abs{N'/d'}])$, which violates Proposition \ref{prop:descimages} as $N'$ is odd and squarefree.
    
    We have seen that $\sqrt{N'} \in \Q(E[2^j]) \cap \Q(E[\abs{N'}])$ yet $\sqrt{N'} \not\in \Q(E[2^i]) \cap \Q(E[d])$ for any $i \leq j$ and $d$ a divisor of $\abs{N'}$ with $(i,d) \neq (j,\abs{N'})$. Thus $G_E$ satisfies the conditions of Lemma \ref{L:imgcondmult} with $(m_1,m_2) = (2^j,\abs{N'})$. The lemma  gives that $2^j \cdot \abs{N'}$ divides $m_E$ as claimed. Thus $\mathcal{L}$ divides $m_E$, completing the proof of the considered case. The remaining cases follow similarly.
\end{proof}

\subsection{Serre curves relative to 2B}\label{subsec:2B}
Let $E$ be a \TwoB-Serre curve. Fix an equation
\begin{equation} E \colon y^2 = (x^2+ax+b)(x+c) \label{E:2B} \end{equation}
 with $a,b,c \in \ZZ$ and $x^2 + ax + b$ irreducible over $\Q$. By Proposition \ref{P:adelic_index}, we know that
\[ [\GL_2(\widehat{\ZZ}):G_E] = 12.\] 
Recall that 
\[\mathcal{S}_{\TwoB} = \{\six, \fourteen,\fifteen,\sixteen,\seventeen,\eighteen,\nineteen\}.\] Consequently,
\begin{equation} \label{E:2B2adic}
[\GL_2(\ZZ_2):G_E(2^\infty)] = 
\begin{cases}
    3 & G_E(2^\infty) = \six \\
    6 & \text{otherwise}.
\end{cases}
\end{equation}
Hence to understand $G_E$, it remains to account for an index of $4$ if $G_E = \six$ and 2 otherwise. Considering each of the groups in $\mathcal{S}_{\TwoB}$ under reduction modulo $4$, we verify in Magma the following useful observation: If $E$ is any \TwoB-Serre curve, then
\begin{equation} \label{E:2BMod4} 
[\GL_2(\ZZ/4\ZZ) :G_E(4)] = 3 \quad \text{and} \quad \text{$G_E(4)$ contains exactly $7$ index 2 subgroups}.
\end{equation}

The roots of $(x^2+ax+b)(x+c)$ are 
\[
r_1 \coloneqq \frac{1}{2}\paren{\sqrt{a^2-4b}-a}, \quad r_2 \coloneqq \frac{1}{2}\paren{-\sqrt{a^2-4b}-a}, \text{ and } \quad r_3 \coloneqq -c.
\]
Let $A \coloneqq r_1-r_2, B \coloneqq r_2 - r_3,$ and $C \coloneqq r_1-r_3$. From Theorem \ref{T:Yelton}, we now deduce an alternative expression for the 4-division field of $E$ in terms of the notation just introduced.

\begin{lemma}\label{torsion 4 2B}
With notation as above, if $E$ is a \TwoB-Serre curve, then \[\Q(E[4])=\Q(\sqrt{-1},\sqrt{a^2-4b},\sqrt{ac-c^2-b},\sqrt{A},\sqrt{B}).\]
\end{lemma}

\begin{proof} By Theorem \ref{T:Yelton}\eqref{T:Yelton2}, we have that
\[ \Q(E[4]) = \Q(\sqrt{-1},\sqrt{A},\sqrt{B},\sqrt{C}). \]
We have that $\sqrt{C} \in \Q(\sqrt{-1},\sqrt{a^2-4b},\sqrt{ac-c^2-b},\sqrt{A},\sqrt{B})$ because
\begin{equation} \sqrt{B}\sqrt{C} = \sqrt{-1} \sqrt{ac- b - c^2}. \label{torsion 4 2B-1}\end{equation}
It follows that
\[\Q(E[4]) \subseteq \Q(\sqrt{-1},\sqrt{a^2-4b},\sqrt{ac-c^2-b},\sqrt{A},\sqrt{B}).\] 
The reverse inclusion follows by \eqref{torsion 4 2B-1} and the fact that $A = \sqrt{a^2-4b}$.
\end{proof}

Recall that by \eqref{E:2BMod4}, there are exactly seven distinct quadratic subfields of $\Q(E[4])$. Using Lemma \ref{torsion 4 2B}, we now describe them explicitly.

\begin{lemma}\label{lemma 3-6} With notation as above, if $E$ is a \TwoB-Serre curve, then $\Q(E[4])$ has exactly 7 distinct quadratic subfields $\colon$  $\Q(\sqrt{-1})$, $\Q(\sqrt{\pm \Delta_E})$, $\Q(\sqrt{\pm (ac-c^2-b)})$, and $\Q(\sqrt{\pm (ac-c^2-b)\Delta_E}).$ 
\end{lemma}

\begin{proof}
By \eqref{E:DeltaE}, we note that $\Q(\sqrt{\Delta_E}) = \Q(\sqrt{a^2-4b})$. Thus it suffices to show that \[[\Q(\sqrt{-1},\sqrt{a^2-4b},\sqrt{ac-c^2-b}) :\Q]=8.\]
We have that $A = \sqrt{a^2 - 4b}$ and $B = \frac{1}{2}\paren{-\sqrt{a^2-4b}-a} + c$. Thus
\[
\sqrt{a^2 - 4b} \in \Q(\sqrt{A}) \cap \Q(\sqrt{B}),
\]
so in particular $[\Q(\sqrt{A}) \cap \Q(\sqrt{B}) : \Q] \geq 2$. Hence it follows by  \eqref{E:CompIndx} that

\[
[\Q(\sqrt{A},\sqrt{B}) : \Q] = \frac{[\Q(\sqrt{A}) : \Q][\Q(\sqrt{B}) : \Q]}{[\Q(\sqrt{A}) \cap \Q(\sqrt{B}) : \Q]} \leq \frac{4 \cdot 4}{2} = 8
\]

Thus, by \eqref{E:CompIndx},\eqref{E:2BMod4}, and  Lemma \ref{torsion 4 2B} ,
\[
32 = [\Q(E[4]) : \Q] 
\leq \frac{[\Q(\sqrt{-1},\sqrt{a^2-4b},\sqrt{ac-c^2-b}) : \Q] \cdot 8}{2}
\]
Hence
\[ [\Q(\sqrt{-1},\sqrt{a^2-4b},\sqrt{ac-c^2-b}) : \Q] \geq 8. \]
The reverse inequality is clear.
\end{proof}

We now consider entanglements. Define the set
\[ S\coloneqq\{ \abs{(a^2-4b)_{\text{sf}}}, \abs{(ac-c^2-b)_{\text{sf}}}, \abs{((ac-c^2-b)(a^2-4b) )_{\text{sf}}} \} .\]
If $G_E(2^{\infty}) = \six$, then let $N_1$ be the smallest integer in $S$ and $N_2$ be the smallest integer in $S$ such that $N_1 \nmid N_2.$ Using Corollary \ref{C:sqrt2} we note that $\sqrt{2},\sqrt{-2} \not\in \Q(E[4])$, since $\# \six(8) = 16 \cdot \# \six(4)$, so $N_i' \neq 1$. We get a horizontal entanglement between $\Q(E[2^{k_i}])$ and $\Q(E[\abs{N_i'}])$ for each $i \in \{1,2\}.$  If $G_E(2^{\infty}) \neq \six$, then let $N_1$ be the smallest integer in $S$ such that $N_1' \neq 1$. In this case, we get one horizontal entanglement between $\Q(E[2^{k_1}])$ and $\Q(E[\abs{N_1'}]).$ 

\subsubsection{Description of image} We now describe the adelic image $G_E$ of a \TwoB-Serre curve $E$. Let $N_i$ be as above and let $k_i$ be as in \eqref{E:ki}. For each $i$, let $G_i \subseteq G_E(2^{k_i})$ be such that $G_E(2^{k_i})/G_i \cong \Gal(\Q(\sqrt{N_i'})/\Q) \cong \{\pm 1\}$.  Let $\epsilon_i \colon \widehat{G_E(2^{\infty})} \to G_E(2^{\infty}) \to G_E(2^{\infty})/G_i$ be the natural map and $\chi_{N_i'}$ be the Dirichlet character given by the unique surjective homomorphism $\widehat{\ZZ}^\times \to \Gal(\Q(\sqrt{N_i'})/\Q) \cong \set{\pm 1}.$ 

\begin{proposition} \label{P:TwoB}
With notation as above, if $G_E(2^\infty)=\six$, then   
\[G_E= \bigcap\limits^2_{i=1}\{A \in \widehat{G_E(2^\infty)}|\epsilon_i(A)=\chi_{N_i'}(\det(A))\}\] for $i \in \{1,2\}.$
If $G_E(2^\infty) \neq \six$, then   
\[G_E=\{A \in \widehat{G_E(2^\infty)}|\epsilon(A)=\chi_{N_1'}(\det(A))\}.\] 
\end{proposition}

\begin{proof}

Assume that $G_E(2^\infty)=\six$. The inclusion \begin{equation}
    \label{PE:ClaimedEq2}
G_E \subseteq \bigcap\limits^2_{i=1}\{A \in \widehat{G_E(2^\infty)}|\epsilon_i(A)=\chi_{N_i'}(\det(A))\}
\end{equation} follows from the containments $\Q(\sqrt{N_i'}) \subseteq \Q(E[2^{k_i}])$ and $\Q(\sqrt{N_i'}) \subseteq \Q(E[\abs{N_i'}]).$ By Proposition \ref{P:adelic_index}, to establish that \eqref{PE:ClaimedEq2} is actually an equality, it suffices to note that 
\begin{equation} \label{PE:2BAdelicIndx} \left[ \GL_2(\Zhat) : 
\bigcap\limits^2_{i=1}\{A \in \widehat{G_E(2^\infty)}|\epsilon_i(A)=\chi_{N_i'}(\det(A))\} \right] = 12. \end{equation} 
From \eqref{E:2B2adic}, we have that
\[ \left[ \GL_2(\widehat{\ZZ}) : \widehat{G_E(2^\infty)} \right] 
= \left[ \GL_2(\ZZ_2) : G_E(2^\infty) \right] = 3. \] Similar to the proof given in Proposition \ref{prop:descimages}, we can show that \[ \left[ \widehat{G_E(2^\infty)} : \bigcap\limits^2_{i=1}\{A \in \widehat{G_E(2^\infty)}|\epsilon_i(A)=\chi_{N_i'}(\det(A))\} \right] = 4. \]
Thus \eqref{PE:2CsAdelicIndx} holds as claimed. When $G_E(2^\infty) \neq \six$ the proof is similar. 
\end{proof}

From the above proposition, we obtain the following formula for $m_E$ upon noting that the level of $\six$ is 2, and otherwise if  $H \in \mathcal{S}_{\TwoB}$, then the level of $H$ is 8.

\begin{corollary} Let $E/\Q$ be a \TwoB-Serre curve.
\begin{itemize}
    \item If $G_E(2^\infty)=\six$, then\[
    m_E = \begin{cases}
        \lcm(4,\abs{N_1'},\abs{N_2'}) & 2 \nmid N_1 N_2 \\
        \lcm(8,\abs{N_1'},\abs{N_2'}) & 2 \mid N_1 N_2 . \\
    \end{cases}
    \]
    
    \item If $G_E(2^\infty) \neq {\six}$, then\[
    m_E = 
        \lcm(8,\abs{N_1'}).
    \]
\end{itemize}
\end{corollary}

\begin{proof}
The proof is similar to the proof of Corollary \ref{cor:m_E}.
\end{proof}

\subsection{Serre curves relative to 2Cn} Let $E$ be a \TwoCn-Serre curve. By Theorem \ref{T:MainTheorem}, we know that $G_E(2^{\infty}) \in \mathcal{S}_{\TwoCn}$. In particular, $G_E(2) = \TwoCn$. Thus, the Galois extension $\Q(E[2])/\Q$ is cyclic of order $3$. By \eqref{E:cond_cubic}, the conductor of $\Q(E[2])$ is given by
\begin{equation} f(\Q(E[2])) = \sqrt{\Delta_{\Q(E[2])}}. \label{E:cubic-conductor} \end{equation}
Further, we claim that $f(\Q(E[2]))$ is odd. Indeed, a cubic extension of $\Q$ is given by $\Q(\alpha)$ where $\alpha$ satisfies the irreducible polynomial $T^3-3T+1-v(T^2-T)$ for some $v \in \Q$ (see, e.g., \cite[Section 1.1]{MR2363329}). The square root of discriminant of $\Q(\alpha)$ is either $v^2-3v+9$ or $(v^2-3v+9)/3$. Since we know that it has to be an integer, $v \in \ZZ$, and the claim follows.

By Proposition \ref{P:adelic_index}, we have that
\[ [\GL_2(\widehat{\ZZ}) : G_E] = 12.\] 
Recall that $\mathcal{S}_{\TwoCn} = \{\two, \twoa, \twob\}$. Thus, considering the labels,
\[
[\GL_2(\ZZ_2): G_E(2^\infty)] = 
\begin{cases}
    2 & G_E(2^\infty) = \two \\
    4 & \text{otherwise}.
\end{cases}
\]

If $G_E(2^{\infty}) \neq \two$, then the adelic index of $12$ is explained by the fact that $[\GL_2(\ZZ_2) : G_E(2^{\infty})]=4$ and the cubic entanglement arising from the containment $\Q(E[2]) \subseteq \Q(\zeta_{\sqrt{\Delta_{\Q(E[2])}}})$. Consider the case that  $G_E(2^{\infty}) = \two$. Here $[\GL_2(\ZZ_2) : G_E(2^{\infty})]=2$, so an additional quadratic entanglement must be described to account for the adelic index of 12. By \eqref{E:cubic-conductor}, we know that  $D_E \coloneqq (\sqrt{\Delta_E})_{sf} \in \ZZ.$ We have that $\sqrt{D_E}$ lies in $\Q(E[4])$ and is not equal to $\sqrt{\pm 2}$ because if $\sqrt{2} \in \Q(E[4])$, then $[\Q(E[8]) : \Q(E[4])]$ is at most $8$ by Corollary \ref{C:sqrt2}, which is a contradiction since we know that $\# \two(8) = 16 \cdot \# \two(4)$. Using the notation \eqref{E:prime-notation}, there is a quadratic entanglement between $\Q(E[2^k])$ and $\Q(E[\abs{D_E'}])$, where $k$ is as defined in \eqref{E:ki}.

\subsubsection{Description of image}  We now describe the adelic image $G_E$ of a \TwoCn-Serre curve $E$. Let $G_1 \subseteq G_E(2^{\infty})$ be the index 3 subgroup such that $G_E(2^{\infty})/G_1 \cong \Gal(\Q(E[2])/\Q)$. Let $\omega \colon \widehat{G_E(2^\infty)} \to G_E(2^{\infty}) \to G_E(2^{\infty})/G_1$ be the natural map. Let $\xi \colon \Zhat^{\times} \to \Gal(\Q(E[2])/\Q)$ be the restriction of cyclotomic character associated to $\Q(\zeta_{\sqrt{\Delta_{\Q(E[2])}}}).$ If $G_E(2^\infty) = \two$, further let $G_2 \subseteq G_E(2^k)$ be the index $2$ subgroup such that $G_E(2^k)/G_2 \cong \Gal(\Q(\sqrt{D_E'})/\Q) \cong \{\pm 1\}$, where $k$ is $2$ if $D_E$ is odd and $k$ is $3$ otherwise. Let $\epsilon \colon \widehat{G_E(2^\infty)} \to G_E(2^k) \to G_E(2^k)/G_2$ be the natural map. Let $\chi$ be the Dirichlet character given by the unique surjective homomorphism $\widehat{\ZZ}^\times \to \Gal(\Q(\sqrt{D_E'})/\Q) \cong \{\pm 1\}.$

\begin{proposition} \label{P:2Cn}
With notation as above, if $G_E(2^\infty)=\two$, then   
\[G_E=\{A \in \widehat{G_E(2^\infty)} |\omega(A)=\xi(\det(A))\} \cap \{A \in \widehat{G_E(2^\infty)} |\epsilon(A)=\chi(\det(A))\}.\]
If $G_E(2^\infty) \in \set{\twoa,\twob}$, then   
\[G_E=\{A \in \widehat{G_E(2^\infty)} |\omega(A)=\xi(\det(A))\}.\]
\end{proposition}

\begin{proof}

Assume that $G_E(2^\infty)=\two$. The inclusion 
\begin{equation} \label{PE:ClaimedEqCn}
G_E \subseteq \{A \in \widehat{G_E(2^\infty)}|\omega(A)=\xi(\det(A))\} \cap \{A \in \widehat{G_E(2^\infty)}|\epsilon(A)=\chi(\det(A))\}
\end{equation}
follows from $\Q(E[2]) \subseteq \Q(E[8])$, 
$\Q(E[2]) \subseteq \Q(\zeta_{\sqrt{\Delta_{\Q(E[2])}}}) \subseteq \Q(E[\sqrt{\Delta_{\Q(E[2])}}])$ and $\Q(\sqrt{D_E'}) \subseteq \Q(E[4]) \cap \Q(E[\abs{D_E'}]).$
Similar to the proof given in Proposition \ref{prop:descimages} and by Proposition \ref{P:adelic_index}, we have that \eqref{PE:ClaimedEqCn} is actually an equality. When $G_E(2^\infty) \neq \two$, the proof is similar. 
\end{proof}

From the above proposition, we obtain the following formula for $m_E$ upon noting that the level of $\two$ is 2, $\twoa$ is $4$, and $\twob$ is 8.

\begin{corollary} Let $E/\Q$ be a \TwoCn-Serre curve.
\begin{itemize}
    \item If $G_E(2^\infty)=\twoa$, then\[
    m_E = 
        \lcm(4,\sqrt{\Delta_{\Q(E[2])}}).\]
    
    \item If $G_E(2^\infty)=\twob$, then\[
    m_E = 
        \lcm(8,\sqrt{\Delta_{\Q(E[2])}}).\]
    
    \item If $G_E(2^\infty) = \two$, then\[
    m_E = 
    \begin{cases}
        \lcm(4,\sqrt{\Delta_{\Q(E[2])}},\abs{D_E'}) & 2 \nmid D_E  \\
        \lcm(8,\sqrt{\Delta_{\Q(E[2])}},\abs{D_E'}) & 2 \mid D_E. \\
    \end{cases}
    \]
\end{itemize}

\end{corollary}

 \begin{proof}
    The proof is similar to the proof of Corollary \ref{cor:m_E}. \end{proof}

\section{An application} \label{S:Application}

In the previous section, we described the adelic Galois image of a Serre curve relative to an obstruction modulo $2$. There are numerous applications of such an understanding, including to problems concerning the distribution of prime numbers with certain properties relating to the arithmetic of elliptic curves. The following problems are some well-known examples of this sort.
\begin{enumerate}
    \item Koblitz conjecture \cite{MR917870} (and Zywina's refinement \cite{MR2805578})
    \item Lang-Trotter conjecture \cite{MR0568299}
    \item Titchmarsh divisor problem for elliptic curves \cite{MR3432335}
    \item Cyclicity conjecture \cite{MR3223094}
\end{enumerate}

In this section, we consider the cyclicity conjecture, which asks: For an elliptic curve $E/\Q$, what is the density of primes $p$ of good reduction for $E$ with the property that the group $E(\F_p)$ is cyclic? Serre first studied this problem. He proved the following conditional result.

\begin{theorem}[Serre \cite{MR3223094}] \label{T:SerreCyc} \label{T:Cyclicity} Assume GRH. If $E/\Q$ is an elliptic curve of conductor $N_E$, then
\[
\#\{
p \leq x : p \nmid N_E \text{ and } E(\F_p) \text{ is cyclic}
\}
\sim
C_E \frac{x}{\log x}
\]
as $x \to \infty$, where $C_E \coloneqq \sum_{n \geq 1} \frac{\mu(n)}{[\Q(E[n]) : \Q]}$ in which $\mu(\cdot)$ denotes the M\"obius function.
\end{theorem}

The constant $C_E$, defined in Theorem \ref{T:Cyclicity},  depends only on the adelic image $G_E$.

Assume that $E$ \textit{has abelian entanglements}, by which we mean that for every pair of relatively prime positive integers $(m_1,m_2)$, the extension $\Q(E[m_1]) \cap \Q(E[m_2])$ over $\Q$ is abelian. In his thesis, Brau \cite{BrauThesis,BrauCharacters} gave a framework for explicitly computing $C_E$ (provided that $E$ has abelian entanglements). Let $m \coloneqq (m_E)_{\text{sf}}$ be the squarefree part of the image conductor $m_E$ and consider the quotient group
\[ \Phi_E \coloneqq \left(\prod_{\ell \mid m} G_E(\ell) \right) / G_E(m). \]
This group quantifies the prime-level entanglements of $E$. 

With the above notation and terminology, Brau proved the following result.

\begin{theorem}[Brau {\cite{BrauThesis,BrauCharacters}}]\label{cycthm}
    Let $E/\Q$ be an elliptic curve with abelian entanglements. For a character $\chi$ of the group $\prod_{\ell |m} G_E(\ell)$ obtained by composing the natural quotient map $ \prod_{\ell |m} G_E(\ell) \to \Phi_E$ with a character $\Tilde{\chi}$ of the group $\Phi_E$, define the constant $E_{\chi,\ell}$ as follows:
    \[
    E_{\chi,\ell} = 
    \begin{cases}
        1 & \text{if $\chi$ is trivial on $G_E(\ell)$}  \\
        \frac{-1}{[\Q(E[\ell]) : \Q]-1} & \text{otherwise.} \\
    \end{cases}
    \]
    Then the cyclicity constant of Theorem \ref{T:SerreCyc} is given by
    \[C_{E}=\mathfrak{C}_{E} \prod_{\ell} \left( 1-\frac{1}{[\Q(E[\ell]) : \Q]} \right)\]
    where $\mathfrak{C}_{E}$ denotes the entanglement correction factor
     \[\mathfrak{C}_{E} \coloneqq 1+\sum_{\Tilde{\chi} \in \widehat{\Phi}_E \setminus \{1\} } \prod_{\ell|m} E_{\chi,\ell}\]
     in which $\widehat{\Phi}_E$ denotes the character group of $\Phi_E$.
    \end{theorem}

If $E$ is a Serre curve or if $E$ is a $G$-Serre curve for $G \in \{\TwoCs,\TwoB,\TwoCn \}$, then $E$ has abelian entanglements. Applying Theorem \ref{cycthm} when $E$ is a Serre curve, Brau gave a formula for the cyclicity constant $C_E$ in terms of only the discriminant $\Delta_E$. In a similar way, we now apply Theorem \ref{cycthm} with our results of Section \ref{S:descriptions} to give an analogous formula for relative Serre curves. The case when $E$ is a \TwoCs-Serre curve is trivial, since $G_E(2) = \TwoCs$ implies that $(\ZZ/2\ZZ) \oplus (\ZZ/2\ZZ)$ is contained in the rational torsion of $E$, so $C_E = 0$. Thus we only consider the cases of \TwoB- and \TwoCn-Serre curves. 

\begin{proposition} Let $E$ be a \TwoB-Serre curve given by \eqref{E:2B}. Write $\Delta_{E,\text{sf}}$ to denote the squarefree part of the discriminant of any model of $E$. Then the cyclicity constant is given by
\[C_{E}=\mathfrak{C}_{E} \cdot \frac{1}{2} \prod_{\ell\text{ odd}} \left( 1-\frac{1}{(\ell^2-1)(\ell^2-\ell)} \right)\] where 
        \[\mathfrak{C}_{E}=
        \begin{cases} \displaystyle
        1 - \prod_{\ell|\Delta_{E,\text{sf}}}\frac{-1}{(\ell^2-1)(\ell^2-\ell)} & \Delta_{E,\text{sf}} \equiv 1 \pmod{4} \\
            1 & \text{otherwise}.
        \end{cases}\]
\end{proposition}
\begin{proof} We have that $\Q(E[2]) = \Q(\sqrt{\Delta_{E,\text{sf}}})$. Thus $\Q(E[2])$ entangles with $\Q(E[d])$ for some $d$ coprime to $2$ if and only if $\Delta_{E,\text{sf}} \equiv 1 \pmod 4$. By this fact, together with Proposition \ref{P:TwoB}, we have that $\Phi_E$ is of order 2 if $\Delta_{E,\text{sf}} \equiv 1 \pmod 4$ and $\Phi_E$ is trivial otherwise. Consider the former case. Here there are $2$ characters of the group $\Phi_E$. Moreover, the non-identity character $\chi$ corresponds to the non-identity character of group $\Phi_E$ and is non-trivial on $G_E(\ell)$ for each $\ell|m$ because it factors through the quotient group $\Phi_E$. So, the entanglement correction factor must be given by 
\begin{equation} \label{E:CE4.3} \mathfrak{C}_{E}=1 + (-1)\prod_{\ell|\Delta_{E,\text{sf}}}\frac{-1}{(\ell^2-1)(\ell^2-\ell)}\end{equation}
and cyclicity constant is given by 
\[C_{E}=\mathfrak{C}_{E} \frac{1}{2} \prod_{\ell\text{ odd}} \left( 1-\frac{1}{(\ell^2-1)(\ell^2-\ell)} \right).\]
since $[\Q(E[\ell]) : \Q] = \#G_E(\ell) = 2$ if $\ell = 2$ and $[\Q(E[\ell]) : \Q] = \#\GL_2(\ZZ/\ell\ZZ) = (\ell^2-1)(\ell^2-\ell)$ otherwise. Note the $-1$ before the product in \eqref{E:CE4.3} comes from the prime $\ell=2$. In the latter case, that $\Delta_{E,\text{sf}} \not\equiv 1 \pmod 4$, we have that $\mathfrak{C}_E = 1$ as $\Phi_E$ is trivial.
\end{proof}

\begin{proposition}
    Let $E$ be a \TwoCn-Serre curve. Then the cyclicity constant is given by \[C_{E}=\mathfrak{C}_{E} \cdot \frac{2}{3} \prod_{\ell\text{ odd}} \left( 1-\frac{1}{(\ell^2-1)(\ell^2-\ell)} \right) \] where \[\mathfrak{C}_{E}= 
        \begin{cases} \displaystyle
            1 - \prod_{\ell| \sqrt{\Delta_{\Q(E[2])}}}\frac{-1}{(\ell^2-1)(\ell^2-\ell)} & \text{ if }\sqrt{\Delta_{\Q(E[2])}} \text{ is squarefree}\\
            1 & \text{otherwise}.
        \end{cases}.\]
    \end{proposition}
    
    \begin{proof} We have that $\Q(E[2])$ is a cyclic cubic field of conductor $\sqrt{\Delta_{\Q(E[2])}}$. By this fact, together with Proposition \ref{P:2Cn}, we have that $\Phi_E$ is of order 3 if  $\sqrt{\Delta_{\Q(E[2])}}$ is squarefree, and $\Phi_E$ is trivial otherwise. Consider the former case. Here there are 3 characters of the group $\Phi_E$. Moreover, the non-identity characters $\chi$ correspond to non-identity characters of $\Phi_E$, and are non-trivial on $G_E(\ell)$ for each $\ell|m$ because they factor through $\Phi_E$. Thus, the entanglement correction factor of Theorem \ref{cycthm} is given by \[\mathfrak{C}_{E}=1 + 2(-1/2)\prod_{\ell|\sqrt{\Delta_{\Q(E[2])}}}\frac{-1}{(\ell^2-1)(\ell^2-\ell)}\] and cyclicity constant is given by \[C_{E}=\mathfrak{C}_{E} \cdot \frac{2}{3} \prod_{\ell\text{ odd}} \left( 1-\frac{1}{(\ell^2-1)(\ell^2-\ell)} \right)\]
    since $[\Q(E[\ell]) : \Q] = \#G_E(\ell)$ equals $3$ if $\ell = 2$ and equals $\#\GL_2(\ZZ/\ell\ZZ)$ otherwise.
    \end{proof}

\section{Some examples}\label{S:examples}

In this section, we build on Section \ref{S:descriptions} to explicitly compute the adelic Galois image of three relative Serre curves. In each case, our result agrees (up to conjugacy) with the output of Zywina's recent general algorithm \cite{ZywinaOpen}. The code that accompanies this section is available on this paper's GitHub repository \cite{MayleRakviGitHub}.

\begin{example}\label{Ex1} Consider the elliptic curve $E$ with LMFDB label 315.a2 given by the Weierstrass equation 
\[y^2=x^3-1083x+10582=(x-11)(x-26)(x+37).\] 
Using our \texttt{IsRelSerreCurve} script, we find that $E$ is a $\TwoCs$-Serre curve. With the notation of Section \ref{subsec:2Cs}, we have that
\[ S=\{\pm 3, \pm 5, \pm 7, \pm 15, \pm 21, \pm 35 \}. \]
Thus $(N_1,k_1)=(3,2), (N_2,k_2)=(5,2)$, and $(N_3,k_3)=(7,2)$. Hence $N_1'=-3,N_2'=5$, and $N_3'=-7.$ There is a quadratic entanglement between $\Q(E[4])$ and $\Q(E[\abs{N_i'}])$ for each $i \in \{1,2,3\}$. Using Sutherland's \texttt{galrep} \cite{MR3482279}, we compute
\[
G_E(4) = \left\langle
\begin{psmallmatrix}
    3 & 0 \\
    0 & 1
\end{psmallmatrix},
\begin{psmallmatrix}
    1 & 2 \\
    0 & 3
\end{psmallmatrix},
\begin{psmallmatrix}
    1 & 2 \\
    2 & 1
\end{psmallmatrix},
\begin{psmallmatrix}
    3 & 0 \\
    0 & 3
\end{psmallmatrix}
\right\rangle.
\]
Let $G_i$ be the index 2 subgroup of $G_E(4)$ corresponding to $\Gal(\Q(E[4])/\Q(\sqrt{N_i'})).$ By constructing the field $\Q(E[4])$ in Magma and considering its automorphisms that fix the subfield $\Q(\sqrt{N_i'})$, we determine that
\begin{align*}
G_1 &= \left\langle
\begin{psmallmatrix}
    1 & 2 \\
    0 & 1
\end{psmallmatrix}, \begin{psmallmatrix}
    1 & 0 \\
    0 & 3
\end{psmallmatrix}, \begin{psmallmatrix}
    1 & 2 \\
    2 & 3
\end{psmallmatrix}, \begin{psmallmatrix}
    1 & 2 \\
    0 & 3
\end{psmallmatrix}, \begin{psmallmatrix}
    1 & 2 \\
    2 & 1
\end{psmallmatrix}, \begin{psmallmatrix}
    1 & 0 \\
    2 & 1
\end{psmallmatrix}, \begin{psmallmatrix}
    1 & 0 \\
    2 & 3
\end{psmallmatrix} \right\rangle, \\
G_2 &= \left\langle
\begin{psmallmatrix}
    1 & 2 \\
    2 & 3
\end{psmallmatrix}, \begin{psmallmatrix}
    1 & 2 \\
    0 & 3
\end{psmallmatrix}, \begin{psmallmatrix}
    3 & 0 \\
    2 & 3
\end{psmallmatrix}, \begin{psmallmatrix}
    3 & 0 \\
    0 & 3
\end{psmallmatrix}, \begin{psmallmatrix}
    3 & 2 \\
    0 & 1
\end{psmallmatrix}, \begin{psmallmatrix}
    1 & 0 \\
    2 & 1
\end{psmallmatrix}, \begin{psmallmatrix}
    3 & 2 \\
    2 & 1
\end{psmallmatrix} \right\rangle, \\
G_3 &= \left\langle \begin{psmallmatrix}
    1 & 2 \\
    0 & 1
\end{psmallmatrix}, \begin{psmallmatrix}
    1 & 0 \\
    0 & 3
\end{psmallmatrix}, \begin{psmallmatrix}
    1 & 2 \\
    0 & 3
\end{psmallmatrix}, \begin{psmallmatrix}
    3 & 0 \\
    2 & 1
\end{psmallmatrix}, \begin{psmallmatrix}
    3 & 0 \\
    2 & 3
\end{psmallmatrix}, \begin{psmallmatrix}
    3 & 2 \\
    2 & 3
\end{psmallmatrix}, \begin{psmallmatrix}
    3 & 2 \\
    2 & 1
\end{psmallmatrix} \right\rangle.
\end{align*}
Let $\epsilon_i : G_E(4) \to G_E(4)/G_i$ be the natural map and $\chi_i : \GL_2(\ZZ/\abs{N_i'}\ZZ) \to \{\pm1\}$ be the unique quadratic character (which is given by the determinant map composed with the Kronecker symbol). The image conductor of $E$ is $420 = \lcm(4,3,5,7)$. In view of Proposition \ref{prop:descimages}, we have that
\[
G_E(420) = \{ g \in \GL_2(\ZZ/420\ZZ) : \pi_4(g) \in G_E(4) \text{ and } \epsilon_i(g) = \chi_i(g) \text{ for each $i$} \}.
\]
Using Magma, we compute this fiber product to be the group
\[ G_E(420) = \left\langle \begin{psmallmatrix}
    259 & 362 \\
    162 & 365
\end{psmallmatrix}, \begin{psmallmatrix}
    401 & 108 \\
    364 & 275
\end{psmallmatrix}, \begin{psmallmatrix}
    55 & 142 \\
    52 & 45
    \end{psmallmatrix}, \begin{psmallmatrix}
    27 & 184 \\
    40 & 201
\end{psmallmatrix}\right\rangle. \]
Now the adelic image $G_E$ is the full preimage of $G_E(420)$ in $\GL_2(\widehat{\ZZ})$.
\end{example}

\begin{example} Consider the elliptic curve $E$ with LMFDB label 69.a1 given by
\[y^2=x^3-20115x-1094418=(x^2 - 78x - 14031)(x+78).\] 
This curve is a $\TwoB$-Serre curve. With the notation of Section \ref{subsec:2B}, we have that
\[ S=\{\pm 3, \pm 23, \pm 69 \}. \]
Thus $(N_1,k_1)=(3,2)$ and $(N_2,k_2)=(23,2)$, so  $N_1'=-3$ and $N_2'=-23$. Hence there is a quadratic entanglement between $\Q(E[4])$ and $\Q(E[\abs{N_i'}])$ for each $i \in \{1,2\}$. Using \texttt{galrep}, we compute that
\[
G_E(4) = \left\langle
\begin{psmallmatrix}
    3 & 2 \\
    3 & 3
\end{psmallmatrix},
\begin{psmallmatrix}
    1 & 2 \\
    1 & 1
\end{psmallmatrix},
\begin{psmallmatrix}
    1 & 0 \\
    3 & 1
\end{psmallmatrix},
\begin{psmallmatrix}
    1 & 2 \\
    3 & 3
\end{psmallmatrix}
\right\rangle.
\]
Let $G_i$ be the index 2 subgroup of $G_E(4)$ corresponding to $\Gal(\Q(E[4])/\Q(\sqrt{N_i'})).$ As in Example \ref{Ex1}, we determine that
\[
G_1 = \left\langle
\begin{psmallmatrix}
    1 & 0 \\
    0 & 3
\end{psmallmatrix}, \begin{psmallmatrix}
    3 & 0 \\
    0 & 3
\end{psmallmatrix}, \begin{psmallmatrix}
    1 & 0 \\
    1 & 1
\end{psmallmatrix}\right\rangle
\quad \text{and} \quad
G_2 = \left\langle
\begin{psmallmatrix}
    3 & 0 \\
    0 & 3
\end{psmallmatrix}, \begin{psmallmatrix}
    1 & 2 \\
    0 & 1
\end{psmallmatrix},  
\begin{psmallmatrix}
    3 & 0 \\
    3 & 1
\end{psmallmatrix} \right\rangle.
\]
Let $\epsilon_i : G_E(4) \to G_E(4)/G_i$ be the natural map and $\chi_i : \GL_2(\ZZ/\abs{N_i'}\ZZ) \to \{\pm1\}$ be the unique quadratic character. In view of Proposition \ref{P:TwoB}, we have that
\[
G_E(276) = \{ g \in \GL_2(\ZZ/276\ZZ) : \pi_4(g) \in G_E(4) \text{ and } \epsilon_i(g) = \chi_i(g) \text{ for each $i$} \}.
\]
Using Magma, we compute that
\[ G_E(276) = \left\langle \begin{psmallmatrix}
    43 & 128 \\
    65 & 53
\end{psmallmatrix}, \begin{psmallmatrix}
    13 & 140 \\
    63 & 269
\end{psmallmatrix}, \begin{psmallmatrix}
    167 & 26 \\
    23 & 45
    \end{psmallmatrix}, \begin{psmallmatrix}
    129 & 176 \\
    155 & 145
\end{psmallmatrix},\begin{psmallmatrix}
    175 & 188 \\
    182 & 189
\end{psmallmatrix}\right\rangle. \]
Now the adelic image $G_E$ is the full preimage of $G_E(276)$ in $\GL_2(\widehat{\ZZ})$.
\end{example}

\begin{example} Consider the elliptic curve $E$ with LMFDB label 392.a1 given by
\[y^2=x^3 - 7x + 7.\] 
We have that $2 \sqrt{7} = \sqrt[4]{\Delta_E} \in \Q(E[4])$. As $\sqrt{-1} \in \Q(E[4])$, this implies that $\sqrt{-7} \in \Q(E[4]) \cap \Q(E[7])$. Thus there is a quadratic entanglement between $\Q(E[4])$ and $\Q(E[7])$. Further $f(\Q(E[2])) = 7$, so there is a cubic entanglement between $\Q(E[2])$ and $\Q(E[7])$ (and hence also between $\Q(E[4])$ and $\Q(E[7])$). Using \texttt{galrep}, we compute that
\[
G_E(4) = \left\langle
\begin{psmallmatrix}
    2 & 3 \\
    3 & 3
\end{psmallmatrix},
\begin{psmallmatrix}
    1 & 3 \\
    1 & 0
\end{psmallmatrix},
\begin{psmallmatrix}
    2 & 1 \\
    3 & 1
\end{psmallmatrix}
\right\rangle.
\]
Let $G_1$ be the index 2 subgroup of $G_E(4)$ corresponding to $\Gal(\Q(E[4])/\Q(\sqrt{-7})).$ Let $G_2$ be the unique index 3 subgroup of $G_E(4)$. Using Magma, we find that
\[
G_1 = \left\langle
\begin{psmallmatrix}
    1 & 0 \\
    0 & 3
\end{psmallmatrix}, \begin{psmallmatrix}
    0 & 1 \\
    1 & 1
\end{psmallmatrix} \right\rangle
\quad \text{and} \quad
G_2 = \left\langle
\begin{psmallmatrix}
    1 & 0 \\
    0 & 3
\end{psmallmatrix}, \begin{psmallmatrix}
    1 & 0 \\
    2 & 1
\end{psmallmatrix}, \begin{psmallmatrix}
    3 & 2 \\
    0 & 3
\end{psmallmatrix}, \begin{psmallmatrix}
    1 & 2 \\
    2 & 1
\end{psmallmatrix} \right\rangle.
\]
Let $\epsilon : G_E(4) \to G_E(4)/G_1$ and $\omega: G_E(4) \to G_E(4)/G_2$ be the natural maps. There are two nontrivial cubic characters of $G_E(4)$ (both of which have kernel $G_2$). They are the maps $g \mapsto \omega(g)$ and $g \mapsto \omega(g)^2$. Thus by Proposition \ref{P:2Cn}, there are two possible candidates for $G_E(28)$, namely,
\[
H_1 = \left\langle
\begin{psmallmatrix}
    26 & 23 \\
    1 & 11
\end{psmallmatrix}, \begin{psmallmatrix}
    8 & 21 \\
    7 & 17
\end{psmallmatrix} \right\rangle
\quad\text{and}\quad
H_2 = \left\langle
\begin{psmallmatrix}
    26 & 23 \\
    1 & 19
\end{psmallmatrix}, \begin{psmallmatrix}
    19 & 27 \\
    21 & 12
\end{psmallmatrix}, \begin{psmallmatrix}
    8 & 5 \\
    27 & 21
\end{psmallmatrix} \right\rangle.
\]
To determine the correct candidate, we consider the action of Frobenius. Let $\Frob_3 \in \Gal(\overline{\Q}/\Q)$ be a Frobenius automorphism associated with $p = 3$. Using Magma, we compute that the characteristic polynomial of $\rho_{E,28}(\Frob_3)$ is $x^2 + 3x + 3$. This polynomial does not appear as the characteristic polynomial of any matrix in $H_1$. As such, it must be that $G_E(28) = H_2$. We conclude that the adelic image $G_E$ is the full preimage of $G_E(28) = H_2$ in $\GL_2(\widehat{\ZZ})$.
\end{example}

\bibliographystyle{amsplain}
\bibliography{References}

\newpage
\appendix
\section{Table of relative Serre curves}

The table below gives, for each $G \in \{\TwoCs,\TwoB,\TwoCn\}$ and $H \in \mathcal{S}_{G}$, an elliptic curve $E/\Q$ with minimal conductor among $G$-Serre curves for which $G_E(2^\infty) = H$. For each curve, the LMFDB label and minimal Weierstrass equation is listed, along with the image conductor $m_E$ and entanglement correction factor $\mathfrak{C}_E$. If $G_E(2) = \texttt{2Cs}$, then $C_E = 0$. Thus $\mathfrak{C}_E$ is not given for $\TwoCs$-Serre curves.

\begin{table}[H] \renewcommand{\arraystretch}{1.3}
\begin{tabular}{|llllll|}
\hline
$G$    & $G_E(2^\infty)$ & LMFDB  & Weierstrass equation           & $m_E$ & $\mathfrak{C}_E$ \\ \hline
\TwoCs & \eight         & 315.a2      & $y^2+xy+y=x^3-x^2-68x+182$     & 420      & -      \\ \hline
\TwoCs & $\eighta$        & 1800.c3     & $y^2=x^3-3675x+35750$          & 120      & -      \\ \hline
\TwoCs & $\eightb$        & 1089.j2     & $y^2+xy=x^3-x^2-12546x+173047$ & 132       & -      \\ \hline
\TwoCs & $\eightc$        & 120.a3      & $y^2=x^3+x^2-16x-16$           & 120      & -       \\ \hline
\TwoCs & $\eightd$        & 33.a2       & $y^2+xy=x^3+x^2-11x$           & 132      & -      \\ \hline
\TwoCs & $\thirtyeight$        & 350.b3      & $y^2+xy=x^3-x^2-442x-2784$     & 280      & -   \\ \hline
\TwoCs & $\thirtyeighta$       & 392.d3      & $y^2=x^3-931x-10290$           & 56      & -      \\ \hline
\TwoCs & $\thirtyeightb$       & 3136.p3     & $y^2=x^3-3724x+82320$          & 56      & -      \\ \hline
\TwoCs & $\thirtyeightc$       & 112.b3      & $y^2=x^3-19x-30$               & 56      & -      \\ \hline
\TwoCs & $\thirtyeightd$       & 56.a3       & $y^2=x^3-19x+30$               & 56      & -      \\ \hline
\TwoCs & $\fortysix$        & 198.a2      & $y^2+xy=x^3-x^2-198x+1120$     & 264      & -      \\ \hline
\TwoCs & $\fortysixa$       & 288.b3      & $y^2=x^3-21x-20$               & 24      & -      \\ \hline
\TwoCs & $\fortysixb$       & 576.g2      & $y^2=x^3-84x-160$              & 24      & -      \\ \hline
\TwoCs & $\fortysixc$       & 96.b3       & $y^2=x^3+x^2-2x$               & 24      & -      \\ \hline
\TwoCs & $\fortysixd$       & 66.b2       & $y^2+xy+y=x^3+x^2-22x-49$      & 264       & -      \\ \hline
\TwoB  & $\six$         & 69.a1       & $y^2+xy+y=x^3-16x-25$          & 276      & 1      \\ \hline
\TwoB  & $\fourteen$        & 1152.d1     & $y^2=x^3-216x-864$             & 24      & 1      \\ \hline
\TwoB  & $\fifteen$        & 102.a1      & $y^2+xy=x^3+x^2-2x$            & 136      & $\tfrac{78337}{78336}$      \\ \hline
\TwoB  & $\sixteen$        & 46.a2       & $y^2+xy=x^3-x^2-10x-12$        & 184      & $\tfrac{267169}{267168}$      \\ \hline
\TwoB  & $\seventeen$        & 46.a1       & $y^2+xy=x^3-x^2-170x-812$      & 184      & 1      \\ \hline
\TwoB  & $\eighteen$        & 490.f1      & $y^2+xy=x^3-1191x+15721$       & 56      & 1      \\ \hline
\TwoB  & $\nineteen$        & 102.a2      & $y^2+xy=x^3+x^2+8x+10$         & 136      & 1      \\ \hline
\TwoCn & $\two$         & 392.a1      & $y^2=x^3-7x+7$                 & 28      & $\frac{2017}{2016}$      \\ \hline
\TwoCn & $\twoa$        & 392.c1      & $y^2=x^3-x^2-16x+29$           & 28      & $\frac{2017}{2016}$      \\ \hline
\TwoCn & $\twob$        & 3136.b1     & $y^2=x^3-1372x-19208$          & 56      & $\frac{2017}{2016}$      \\ \hline
\end{tabular}
\end{table}

$ $

\end{document}